\documentclass{amsart}
\makeatletter
\@namedef{subjclassname@2020}{Subject}
\makeatother
\usepackage{amsmath,amsthm,amsfonts,amssymb,mathrsfs,pb-diagram,amscd}
\usepackage{amsmath}
\usepackage{amsxtra}
\usepackage{amscd}
\usepackage{amsthm}
\usepackage{amsfonts}
\usepackage{amssymb}
\usepackage{eucal}
\usepackage{color}
\usepackage[misc]{ifsym}

\allowdisplaybreaks[4]
\numberwithin{equation}{section}

\theoremstyle{plain}
\newtheorem{theorem}{Theorem}[section]
\newtheorem{lemma}[theorem]{Lemma}
\newtheorem{proposition}[theorem]{Proposition}
\newtheorem{corollary}[theorem]{Corollary}
\newtheorem{conjecture}[theorem]{Conjecture}

\theoremstyle{definition}
\newtheorem{definition}[theorem]{Definition}
\newtheorem{example}[theorem]{Example}

\newtheorem{remark}[theorem]{Remark}
\newtheorem*{remark*}{Remark}

\makeatletter              
\let\c@equation\c@theorem  
\makeatother

\newcommand{\rank}{\operatorname{rank}}

\DeclareMathOperator{\Sym}{Sym}

\newcommand{\id}{\operatorname{id}}

\newcommand{\CC}{{\mathbb C}}
\newcommand{\NN}{{\mathbb N}}

\newcommand{\ZZ}{{\mathbb Z}}

\newcommand\eps{\epsilon}

\newcommand{\mhcn}{\mathfrak{H}_n} 
\newcommand{\undla}{\underline{\lambda}}

\def\lam{\lambda}
\def\OP{\mathcal{P}}
\def\Sym{\mathfrak{S}}
\def\CC{\mathscr{C}^{c}}
\def\CP{\mathscr{P}^{c}}

\begin{document}

\title{On the (super)cocenter of Cyclotomic Sergeev algebras}

\author{Shuo Li }
\address{School of Mathematics and Statistics,
Beijing Institute of Technology, Beijing, 100081, P. R. China}
\email{3120225737@bit.edu.cn}

\author[Lei Shi]{Lei Shi\textsuperscript{\Letter}}\thanks{\Letter Lei Shi \qquad Email: leishi202406@163.com}
\address{School of Mathematics and Statistics\\
	Beijing Institute of Technology\\
	Beijing, 100081, P.R. China}
\email{leishi202406@163.com}
\begin{abstract}
We show that cyclotomic Sergeev algebra $\mhcn^{g}$ is symmetric when the level is odd and supersymmetric when the level is even. We give an integral basis for ${\rm Tr}(\mhcn^{g})_{\overline{0}}$, and recover Ruff's result on the rank of ${\rm Z}({\mhcn^{g}})_{\bar{0}}$ when the level is odd. We obtain a generating set of ${\rm SupTr}(\mhcn^{g})_{\overline{0}}$, which gives an upper bound of the dimension of ${\rm Z}({\mhcn^{g}})_{\bar{0}}$ when the level is even.
\end{abstract}



\maketitle
\section{Introduction}

The representation theory of symmetric groups $\Sym_n$ has developed into a large and
important area of mathematics, for example, Lie theory, geometry, topology and so
on. Furthermore, the representation theory of the associated Iwahori Hecke algebras of type $A$ as well as their degenerate and non-degenerate cyclotomic generalizations has been well-studied in the literature, see \cite{A3, K3,Ma} and references therein.

In \cite{Sch}, Schur showed that the study of spin (or projective) representation theory of $\Sym_n$ is equivalent to the study of linear representation of $\Sym_n^-.$
The later is “super-equivalent” to the representation theory of the so-called Sergeev algebra. Nazarov \cite{Na2}  introduced affine Sergeev algebra $\mhcn$ to study the spin (or projective) representations of the symmetric group $\Sym_n$ or equivalently, the representation of Sergeev algebra. The cyclotomic Sergeev algebra $\mhcn^{g}$ was introduced by Brundan and Kleshchev \cite{BK01} in the study of modular branching rules for $\widehat{\Sym}_{n}$. $\mhcn^{g}$ can also be viewed as a super version of the degenerate cyclotomic Hecke algebra. Moreover, Kang, Kashiwara and Tsuchioka \cite{KKT} showed that there is a non-trivial $\mathbb{Z}$-grading on cyclotomic Sergeev algebras $\mhcn^{g}$ using cyclotomic quiver Hecke superalgebras. The later give categorifications of highest weight modules for certain quantum groups or super quantum groups  \cite{KKO1,KKO2}.

For finite Hecke-Cilfford algebra, which is the analogue of non-degenerate version of Sergeev algebra, Wan and Wang \cite[Section 5.2]{WW} introduced a symmetrizing trace form for generic Hecke-Cilfford algebra using irreducible characters. Unfortunately, the symmetrizing trace form in \cite{WW} is only proved to be non-degenerate over the field $\mathbb{K}$ rather over the base ring $\mathbb{Z}[\frac{1}{2},q,q^{-1}]$. On the other hand, we don't know whether there is any symmetrizing trace form on cyclotomic non-degenerate and degenerate Hecke-Cilfford algebra or not. This is one of the motivations of our work. To state our main result, we need following definition, which is inspired by \cite[Section 4.1, 5.1]{WW}.

\begin{definition}\label{symmetric}
	
	Let ${\rm R}$ be an integral domain of characteristic different from 2, and ${\rm A}={\rm A}_{\overline{0}}\oplus {\rm A}_{\overline{1}}$ an ${\rm R}$-superalgebra,
	which is finitely generated projective as ${\rm R}$-module. $|\cdot|: {\rm A} \rightarrow \Bbb Z_{2}$ is the parity map.
	
	(i) The superalgebra ${\rm A}$ is called symmetric if there is an ${\rm R}$-linear map $t:{\rm A} \rightarrow {\rm R}$
	with $t({\rm A}_{\overline{1}})=0$ such that $t(xy)=t(yx)$ for any $x, y\in {\rm A}$ and
	$$\hat{t}: {\rm A}\rightarrow {\rm Hom}_{{\rm R}}({\rm A}, {\rm R}), \quad a \mapsto t(-\cdot a) $$
	is an $({\rm A}, {\rm A})-$superbimodule isomorphism. In this case, we call $t$ a symmetrizing form on ${\rm A}$;
	
	(ii) The superalgebra ${\rm A}$ is called supersymmetric if there is an ${\rm R}$-linear map $t:{\rm A} \rightarrow {\rm R}$
	with $t({\rm A}_{\overline{1}})=0$ such that  $t(xy)=(-1)^{|x| |y|}t(yx)$ for any homogeneous $x, y\in {\rm A}$ and
	$$\hat{t}: {\rm A}\rightarrow {\rm Hom}_{{\rm R}}({\rm A}, {\rm R}), \quad a \mapsto t(-\cdot a) $$
	is an $({\rm A}, {\rm A})-$superbimodule isomorphism. In this case, we call $t$ a supersymmetrizing form on ${\rm A}.$
\end{definition}

The first main result of this paper is following.

\begin{theorem}\label{dengerate}
	
	(i) If the level $d$ is odd, then the cyclotomic sergeev algebra $\mhcn^{g}$ is symmetric;
	
	(ii) If the level $d$ is even, then the cyclotomic sergeev algebra $\mhcn^{g}$ is supersymmetric.
\end{theorem}

Theorem \ref{dengerate} implies that the situation in cyclotomic Sergeev algebra is slightly different from usual cyclotomic  Hecke algebra \cite{MM}. Actually, for a symmetric superalgebra ${\rm A}$, the degree zero part of its center ${\rm Z}({\rm A})_{\bar{0}}$ is isomorphic to the dual of degree zero part of its cocenter ${\rm Tr}({\rm A})_{\bar{0}}$ which is as in the non-super case. In contrast, for a supersymmetric superalgebra ${\rm A}$, the degree zero part of its center ${\rm Z}({\rm A})_{\bar{0}}$ is isomorphic to the dual of degree zero part of its supercocenter ${\rm SupTr}({\rm A})_{\bar{0}}$ (see Subsection \ref{basics} for details). We remark here that in general,  one can not find a symmetrizing form on $\mhcn^{g}$ when $d$ is even, we give an example in Example \ref{counterexamp}.

In \cite{Ru}, Ruff obtained a basis for the degree zero part of the center ${\rm Z}(\mhcn^{g})_{\bar{0}}$ when the level $d$ is odd, which immediately gave a classification of the super-blocks for $\mhcn^{g}$ in this case. However, when the level $d$ is even, it is an open probelm to give a basis or even the dimension of the degree zero part of the center ${\rm Z}(\mhcn^{g})_{\bar{0}}$. Theorem \ref{dengerate} implies that we can work on ${\rm SupTr}(\mhcn^{g})_{\bar{0}}$ which seems to be easier than ${\rm Z}(\mhcn^{g})_{\bar{0}}$.  Recently, the second author and Wan \cite{SW} gave a seperate condition, under which the cyclotomic Sergeev algebra was shown to be semisimple. It is natural to ask whether one can compute the characters of those simple modules in \cite{SW} as in \cite{WW} when $\mhcn^{g}$ is semisimple. To answer this question, an integral basis for the cocenter of cyclotomic Sergeev algebra $\mhcn^{g}$ is essential. These motivate our study of  the ${\rm Tr}(\mhcn^{g})_{\bar{0}}$ and ${\rm SupTr}(\mhcn^{g})_{\bar{0}}$  for cyclotomic Sergeev algebra $\mhcn^{g}$.

Our second main result of this paper constructs an integral basis for degree zero part of the cocenter ${\rm Tr}(\mhcn^{g})_{\bar{0}}$, where we refer the readers to \eqref{simple modules}, \eqref{simple modules of type M},  \eqref{label basis} and Subsection \ref{minimal element} for unexplained notations used here.

\begin{theorem}\label{cocenter}
	Suppose ${\rm R}$ is an integral domain with $2$ invertible. Then ${\rm Tr}(\mhcn^{g})_{\overline{0}}$ is a free ${\rm R}$-module with basis $$\bigl\{{\bf w}_\beta+[\mhcn^{g},\mhcn^{g}]_{\overline{0}}\bigm|\beta\in\widetilde{\CP_n}\bigr\}.
	$$ In particular, $\rank_{{\rm R}}{\rm Tr}(\mhcn^{g})_{\overline{0}}=|\widetilde{\CP_n}|=\begin{cases}|\mathscr{P}^{\mathsf{0},m}_n|, &\text{if $d=2m$ is even};\\
		|\mathscr{P}^{\mathsf{s},m}_n|, &\text{if $d=2m+1$ is odd}.
		\end{cases}$
\end{theorem}

The proof of Theorem \ref{cocenter} uses similar techniques in \cite{HS,HuSS} when dealing with even generators and  techniques in \cite{WW} when dealing with Clifford generators. The linearly independence follows from the semisimple representation theory of generic cyclotomic Sergeev algebras developed in \cite{SW} by the second author of this paper and Wan. Combining with Theorem \ref{dengerate}, we recover Ruff's result on the rank of ${\rm Z}(\mhcn^{g})_{\bar{0}}$(Corollary \ref{center}) when $d$ is odd.  We also obtain a generating set for the degree zero part of the supercocenter ${\rm SupTr}(\mhcn^{g})_{\bar{0}}$, which gives rise to an upper bound of dimension for the center of $\mhcn^{g}$ when $d$ is even (Proposition \ref{upperbound}). We propose a conjecture on the rank of ${\rm SupTr}(\mhcn^{g})_{\bar{0}}$ for arbitrary $d$ and prove this conjecture when $d=1$ (Theorem \ref{supercocenter}).

The content of the paper is organised as follows. In Section 2, we shall introduce some basics on (super)symmetrizing superalgebra and cyclotomic Sergeev algebra, including basis Theorem, Mackey decomposition which will be used in later sections. We also compute the dimension of ${\rm Z}(\mhcn^{g})_{\bar{0}}, \,{\rm Tr}(\mhcn^{g})_{\bar{0}},\,{\rm SupTr}(\mhcn^{g})_{\bar{0}}$ for generic cyclotomic Sergeev algebra $\mhcn^{g}$. In Section 3, we use the Frobenious form in \cite{K2} to prove our main result Theorem \ref{dengerate} by induction on $n$. We also give an explicit formula of the (super)symmetrzing form on certain basis. In Section 4, we first recall some main results proved in \cite{HS} on the minimal length elements in the conjugacy classes of the complex reflection group $W_{d,n}$ in Subsection 4.1. Then we give a new presentation of $\mhcn^{g}$ in Subsection 4.2. Using this new presentation, we derive a basis of ${\rm Tr}(\mhcn^{g})_{\overline{0}}$ and recover Ruff's result on the rank of ${\rm Z}(\mhcn^{g})_{\bar{0}}$  \cite[Theorem 5.61]{Ru} when $d$ is odd in Subsection 4.3. In Section 5, we follow a similar computation to obtain a generating set of ${\rm Sup Tr}(\mhcn^{g})_{\overline{0}}$ which exactly gives a basis when $d=1$. As a result, we derive an upper bound on the dimension of the center when $d$ is even.
\bigskip

\centerline{\bf Acknowledgements}
\bigskip
The research is supported by the Natural Science Foundation of Beijing Municipality (No. 1232017).  Both authors thank Jun Hu for his helpful suggestion and feedback.
\bigskip

\section{Preliminary}
\subsection{Some basics on superalgebra}\label{basics}
Recall that ${\rm R}$ is an integral domain of characteristic different from 2 and ${\rm A}={\rm A}_{\overline{0}}\oplus {\rm A}_{\overline{1}}$ an ${\rm R}$-superalgebra,
which is finitely generated projective as ${\rm R}$-module. $|\cdot|: {\rm A} \rightarrow \Bbb Z_{2}$ is the parity map. Let ${\rm Z}({\rm A})$ be the usual center of ${\rm A}$, ${\rm Tr}({\rm A})={\rm A}/[{\rm A}, {\rm A}]$ be the usual cocenter of ${\rm A}$.
Define the supercenter and supercocenter of ${\rm A}$ respectively as follows.
$${\rm SupZ}({\rm A}):={\text  {\rm R-span}}\{x \in {\rm A} \mid xy=(-1)^{|x| |y|}yx, \forall y\in {\rm A} \},$$
$${\rm SupTr}({\rm A}):={\rm A}/[{\rm A}, {\rm A}]^{-},$$
where $[{\rm A}, {\rm A}]^{-}$ is the ${\rm R}$-span of all supercommutators $[x, y]^{-}:=xy-(-1)^{|x| |y|}yx,$ $x, y \in {\rm A}.$
Notice that $\left({\rm SupZ}({\rm A})\right)_{\overline{0}}= {\rm Z}({\rm A})_{\overline{0}}.$

Then we have
\begin{proposition}\label{sym}
(i) If the superalgebra ${\rm A}$ is symmetric, then there is a $({\rm Z}({\rm A}), {\rm Z}({\rm A}))$-supermodule isomorphism
$${\rm Z}({\rm A}) \cong {\rm Hom}_{{\rm R}}({\rm Tr}({\rm A}), {\rm R}), \quad a \mapsto t(-\cdot a);$$

(ii) If the superalgebra ${\rm A}$ is supersymmetric, then there is a $({\rm SupZ}({\rm A}), {\rm SupZ}({\rm A}))$-supermodule isomorphism
$${\rm SupZ}({\rm A}) \cong {\rm Hom}_{{\rm R}}({\rm SupTr}({\rm A}), {\rm R}), \quad z \mapsto t(-\cdot z).$$
\end{proposition}
\begin{proof}We only prove (ii) since the proof of (i) is similar. By Definition \ref{symmetric}, there is an ${\rm R}$-linear map $t:{\rm A} \rightarrow {\rm R}$ such that \begin{equation}\label{isomorphism}\hat{t}: {\rm A}\rightarrow {\rm Hom}_{{\rm R}}({\rm A}, {\rm R}), \quad a \mapsto t(-\cdot a)
	\end{equation}
	is an $({\rm A}, {\rm A})-$superbimodule isomorphism.
 For each homogeneous $z\in {\rm SupZ}({\rm A})$ and homogeneous $x, y\in {\rm A},$ we have
$$
\begin{aligned}
t(xyz) &= (-1)^{|z||y|}t(xzy),\nonumber \\
       &= (-1)^{|z||y|+|y||xz|}t(y xz)\nonumber \\
       &= (-1)^{|z||y|+|y||x|+|y||z|}t(y xz)\nonumber\\
       &= (-1)^{|x||y|} t(y xz),
\end{aligned}
$$ where in the first equation, we have used the definition of $ {\rm SupZ}({\rm A})$ and in the second equation, we have used the definition of supersymmetrizing form . We deduce that $t(-\cdot z) \in {\rm Hom}_{{\rm R}}({\rm SupTr}({\rm A}), {\rm R})$ by the above displayed equation.

Conversely, if $t(-\cdot z)\in {\rm Hom}_{{\rm R}}({\rm SupTr}({\rm A}), {\rm R})\subset {\rm Hom}_{{\rm R}}({\rm A}, {\rm R})$ for some homogeneous $z \in {\rm A},$ then
$$(-1)^{|x||y|+|x||z|}t(yzx)=(-1)^{|x||yz|}t(yzx)=t(xyz)=(-1)^{|x||y|}t(yxz)$$
for any homogeneous $x, y \in {\rm A},$ where we have used the definition of supersymmetrizing form in the second equation and the definition of  $t(-\cdot z)\in {\rm Hom}_{{\rm R}}({\rm SupTr}({\rm A}), {\rm R})$ in the last equation. We deduce that $$t\biggl(y\bigl((-1)^{|x||z|}zx-zx\bigr)\biggr)=0
$$ for any homogeneous $x, y \in {\rm A}.$ It follows from \eqref{isomorphism} that $xz=(-1)^{|x||z|}zx$, for any homogeneous $x \in {\rm A}$, i.e., $z \in {\rm SupZ}({\rm A}).$
\end{proof}

 The following is the super anologue of \cite[Proposition 2.1 (c)]{SVV}.

\begin{proposition}\label{base change}
Suppose ${\rm R'}$ is another commutative domain with a ring homomorphism ${\rm R} \rightarrow {\rm R'}$. We have $${\rm Tr}({\rm R'}\otimes_{\rm R}{\rm A})\cong {\rm R'}\otimes_{\rm R}{\rm Tr}({\rm A}),\qquad {\rm SupTr}({\rm R'}\otimes_{\rm R}{\rm A})\cong {\rm R'}\otimes_{\rm R}{\rm SupTr}({\rm A}).
$$
	\end{proposition}
	
	\begin{proof} We prove the second isomorphism. The proof of the first isomorphism is similar. First, we have the following diagram with two vertical natural maps $\rho_1,\rho_2$ being surjective and two rows being exact.
		\begin{equation*}
			\begin{CD}
			{\rm R'}\otimes_{\rm R}{[{\rm A}, {\rm A}]^{-}} @>\phi_1>> 	{\rm R'}\otimes_{\rm R}{\rm A} @>\pi_1>> {\rm R'}\otimes_{\rm R}{\rm SupTr}({\rm A})@> >>0\\
			@V \rho_1 V V	@V \id V V @VV \rho_2  V\\
		[{\rm R'}\otimes_{\rm R}{\rm A}, {\rm R'}\otimes_{\rm R}{\rm A}]^{-}@>\phi_2>> 	{\rm R'}\otimes_{\rm R}{\rm A}  @>\pi_2>>   {\rm SupTr}({\rm R'}\otimes_{\rm R}{\rm A})@> >>0
			\end{CD}\quad
		\end{equation*}
		By chasing the diagrm, we deduce that $\rho_2$ is an isomorphism.
		\end{proof}

Throughout this paper, $F$ is an algebracially closed field with ${\text{Char}} F\neq 2$. Supppose $V$ is a superspace over $F$, we use $(\dim V_{\bar{0}}, \dim V_{\bar{1}})$ to denote its superdimension. Let $\mathcal{A}$ be a finite dimensional algebra over $F$. A superalgebra analog of Schur's Lemma (cf. \cite{K2}) states that the endomorphism
algebra $\text{End}_{\mathcal{A}}(M)$ of a finite dimensional irreducible module $\mathcal{A}$-module $M$  is either one dimensional or two dimensional. In the
former case, we call the module $M$ of {\em type }\texttt{M} while in
the latter case the module $M$ is called of {\em type }\texttt{Q}.

\begin{example}\label{examples}
	1). Let $V$ be a superspace with superdimension $(m,n)$ over field $F$, then $\mathcal{M}_{m,n}:={\text{End}}_{F}(V)$ is a simple superalgebra with the unique simple module $V$ of {\em type }\texttt{M}. One can check that $${\text{dim}}_F  {\rm Tr}(\mathcal{M}_{m,n})_{\bar{0}} = {\text{dim}}_F {\rm Z}(\mathcal{M}_{m,n})_{\bar{0}}={\text{dim}}_F {\rm SupTr}(\mathcal{M}_{m,n})_{\bar{0}}=1.$$\\
	2). Let $V$ be a superspace with superdimension $(n,n)$ over field $F$. We define
$\mathcal{Q}_n:=\Biggl\{\left(\begin{matrix}
        &A &B \\
		&-B &A
		\end{matrix}\right) \biggm| A,B\in  M_n\Biggr\}\subset \mathcal{M}_{n,n}$. Then $\mathcal{Q}_n$ is a simple superalgebra with the unique simple module $V$ of {\em type }\texttt{Q}. One can check that $${\text{dim}}_F  {\rm Tr}(\mathcal{Q}_n)_{\bar{0}} = {\text{dim}}_F {\rm Z}(\mathcal{Q}_n)_{\bar{0}}=1,\,\,{\text{dim}}_F {\rm SupTr}(\mathcal{Q}_n)_{\bar{0}}=0.$$
\end{example}

Let $\mathcal{J}(\mathcal{A})$  be the usual (non-super) Jacobson radical of $\mathcal{A}$. We call $\mathcal{A}$ is semisimple if $\mathcal{J}(\mathcal{A})=0$.

\begin{lemma}\cite[Lemma 12.2.9]{K1}\label{semisimple}
		A finite dimensional superalgebra $\mathcal{A}$ is semisimple if and only if it is a direct sum of some simple superalgebras. Moreover, any finite dimensional simple superalgebra is isomorphic to some $\mathcal{M}_{m,n}$ or $\mathcal{Q}_n$.
\end{lemma}

\begin{corollary}\label{semisimple1}
Suppose	$\mathcal{A}$ is semisimple over field $F$, then ${\text{dim}}_F  {\rm Tr}(\mathcal{A})_{\bar{0}}= {\text{dim}}_F {\rm Z}(\mathcal{A})_{\bar{0}}$ is the number of simple modules of $\mathcal{A}$ and ${\text{dim}}_F {\rm SupTr}(\mathcal{A}) $ is the number of simple modules of $\mathcal{A}$ of {\em type }\texttt{M}.
	\end{corollary}
	\begin{proof}
		This follows from Example \ref{examples} and Lemma \ref{semisimple}.
		\end{proof}
\subsection{Cyclotomic Sergeev algebra}
Let ${\rm R}$ be an integral domain with $2$ invertible. For $n\in\mathbb{Z}_+$, the affine Sergeev (or degenerate Hecke-Clifford) algebra $\mhcn$ is
the ${\rm R}-$superalgebra generated by even generators
$s_1,\ldots,s_{n-1},x_1,\ldots,x_n$ and odd generators
$c_1,\ldots,c_n$ subject to the following relations
\begin{align}
s_i^2=1,\quad s_is_j =s_js_i, \quad
s_is_{i+1}s_i&=s_{i+1}s_is_{i+1}, \quad|i-j|>1,\label{braid}\\
x_ix_j&=x_jx_i, \quad 1\leq i,j\leq n, \label{poly}\\
c_i^2=1,c_ic_j&=-c_jc_i, \quad 1\leq i\neq j\leq n, \label{clifford}\\
s_ix_i&=x_{i+1}s_i-(1+c_ic_{i+1}),\label{px1}\\
s_ix_j&=x_js_i, \quad j\neq i, i+1, \label{px2}\\
s_ic_i=c_{i+1}s_i, s_ic_{i+1}&=c_is_i,s_ic_j=c_js_i,\quad j\neq i, i+1, \label{pc}\\
x_ic_i=-c_ix_i, x_ic_j&=c_jx_i,\quad 1\leq i\neq j\leq n.
\label{xc}
\end{align}

For $d \in \Bbb N,$ let $$g(x):=\sum_{\substack{0\leq t\leq \frac{d}{2}, \\ t\in \NN}}a_{d-2t}x^{d-2t}\in {\rm R}[x]$$ such that $a_d=1$. For convenience, we denote $a_{d-k}=0$ for any odd number $0\leq k\leq d$. The cyclotomic Sergeev algebra (or degenerate cyclotomic Hecke-Clifford algebra) $\mhcn^{g}$ is defined as $$\mhcn^{g}:=\mhcn/ I_g,
 $$ where $I_g$ is the two sided ideal of $\mhcn$ generated by $g(x_1)$. Denote $$[n]:=\{1,2,\cdots,n\}.$$
  For an (ordered) subset $I=(i_1<i_2<\cdots<i_k)\subset [n]$, we denote $c_I:=c_{i_1}c_{i_2}\cdots c_{i_k}$. By convention, $c_{\emptyset}:=1$. And for any $\alpha=(\alpha_1,\ldots,\alpha_n)\in\mathbb{Z}_{+}^n,$ we set
$x^{\alpha}:=x_1^{\alpha_1}\cdots x_n^{\alpha_n}.$
We first recall some basic facts on $\mhcn^{g}.$

\begin{lemma}\label{lem. of K.}(\cite[Theorem 15.4.6, Lemma 15.5.1]{K2})

(i) The following elements form an ${\rm R}$-basis for $\mhcn^{g},$
\begin{equation}\label{basis1}\{ x^{\alpha}c_I w \mid \alpha \in \Bbb Z_{+}^{n}, \alpha_{1}, \ldots ,\alpha_{n} < d, I\subset [n], w \in \Sym_n \};\end{equation}

(ii)  $\mathfrak{H}_{n+1}^{g}$ is a free right $\mhcn^{g}$-supermodule with basis
\begin{equation}\label{free}
\{ x_{j}^{a}c_{j}^{b}s_{j} \cdots s_{n} \mid 0 \leq a < d, b\in \Bbb Z_{2}, 1 \leq j \leq n+1 \};
\end{equation}

(iii) As $(\mathfrak{H}_{n}^{g}, \mathfrak{H}_{n}^{g})$-bisupermodules,
\begin{equation}\label{decomposition}
\mathfrak{H}_{n+1}^{g}= \bigoplus_{0 \leq a < d, b\in \Bbb Z_{2}}x_{n+1}^{a}c_{n+1}^{b}\mathfrak{H}_{n}^{g}\oplus \mathfrak{H}_{n}^{g}s_{n}\mathfrak{H}_{n}^{g} .
\end{equation}

\end{lemma}

By above $\eqref{free},$ we have (\cite[Proof of Lemma 15.6.2]{K2})

\begin{corollary}\label{decomposition2} For any $y\in \mathfrak{H}_{n+1}^{g}$, we can write $y$ uniquely as
  \begin{equation*}
y=\sum_{a=0}^{d-1} \left( x_{n+1}^{a}\sigma_{a}+x_{n+1}^{a}c_{n+1}\tau_{a} \right) \\
     +\sum_{a=0}^{d-1}\sum_{j=1}^{n}\left( x_{j}^a s_{j}\cdots s_{n}\mu_{a,j} + x_{j}^a c_{j}s_{j}\cdots s_{n}\nu_{a,j} \right)
\end{equation*} for some $\sigma_{a}, \tau_{a}, \mu_{a,j}, \nu_{a,j}\in \mathfrak{H}_{n}^{g}.$
\end{corollary}


\begin{lemma}\label{Formulae}(\cite[(14.8), (14.9), Lemma 15.6.1]{K2})

(i) For $ 1 \leq i < n$ and $a \geq 1,$ we have

\begin{align*}
s_{i}x_{i}^{a} &= x_{i+1}^{a}s_{i}-\sum_{k=0}^{a-1}(x_{i}^{k}x_{i+1}^{a-1-k}+(-x_{i})^{k}x_{i+1}^{a-1-k}c_{i}c_{i+1}),\nonumber \\
s_{i}x_{i+1}^{a} &= x_{i}^{a}s_{i}+\sum_{k=0}^{a-1}(x_{i}^{k}x_{i+1}^{a-1-k}-x_{i}^{k}(-x_{i+1})^{a-1-k}c_{i}c_{i+1});\nonumber
\end{align*}

(ii) For $1 \leq i \leq n,$ and $a \geq 0,$ we have
$$
\begin{aligned}
s_{n}\cdots s_{i} x_{i}^{a} s_{i}\cdots s_{n} &\in x_{n+1}^{a} +\mathfrak{H}_{n}^{g}s_{n}\mathfrak{H}_{n}^{g} \oplus \bigoplus_{0 \leq k \leq a-2, b\in \Bbb Z_{2}}x_{n+1}^{k}c_{n+1}^{b}\mathfrak{H}_{n}^{g}, \nonumber \\
s_{n}\cdots s_{i} x_{i}^{a}c_{i} s_{i}\cdots s_{n} &\in x_{n+1}^{a}c_{n+1} + \mathfrak{H}_{n}^{g}s_{n}\mathfrak{H}_{n}^{g} \oplus \bigoplus_{0 \leq k \leq a-2, b\in \Bbb Z_{2}}x_{n+1}^{k}c_{n+1}^{b}\mathfrak{H}_{n}^{g}.\nonumber
\end{aligned}
$$

\end{lemma}

{\bf  From now on, we shall fix $d\geq 0$ and  let $m:=\lfloor \frac{d}{2}\rfloor$.  Furthermore, we fix $$g(x)=\sum_{\substack{0\leq t\leq \frac{d}{2}, \\ t\in \NN}}a_{d-2t}x^{d-2t}\in{\rm R}[x]$$ such that $a_d=1.$}
 In the rest of this subsection, we shall recall semisimple representation theory of generic cyclotomic Sergeev superalgebra $\mhcn^{h}$.
 To this end, we need some combinatorics. For $n\in \NN$, let $\mathscr{P}_n$ be the set of partitions of $n$ and denote by $\ell(\mu)$ the number of nonzero parts in the partition $\mu$ for each $\mu\in\mathscr{P}_n$. For $a\in \NN$, we use $\mathscr{P}^m_a$ to denote the set of all $m$-partitions of $a$. Let $\mathscr{P}^\mathsf{s}_a$ be the set of strict partition of $a$.
We define
\begin{equation}\label{simple modules}
\mathscr{P}^{\mathsf{0},m}_n:=\mathscr{P}^{m}_n,\qquad
\mathscr{P}^{\mathsf{s},m}_n:=
\sqcup_{a=0}^{n} \mathscr{P}^\mathsf{s}_a\times \mathscr{P}^{m}_{n-a}.\end{equation}
In convention, for any $\undla\in  \mathscr{P}^{\mathsf{0},m}_{n}$, we write $\undla=(\lambda^{(1)},\cdots,\lambda^{(m)})$ while for any $\undla\in  \mathscr{P}^{\mathsf{s},m}_{n}$, we write $\undla=(\lambda^{(0)},\lambda^{(1)},\cdots,\lambda^{(m)})$, i.e., we shall put the strict partition in the $0$-th component.
Then we define
\begin{equation}\label{simple modules of type M}
\mathscr{MP}^{\bullet,m}_{n}:=\begin{cases}
\mathscr{P}^{\mathsf{0},m}_{n}, &\text{if $\bullet=\mathsf{0};$} \\
\biggl\{(\mu,\underline{\lambda})\biggm|\begin{matrix}\mu  {\text { is a strict partition with even length}};\\
	\underline{\lambda} {\text{ is a $m$-partition}}\end{matrix}\biggr\}\subset \mathscr{P}^{\mathsf{s},m}_n, &\text{if $\bullet=\mathsf{s}$}.
\end{cases}
\end{equation}


Let ${\rm K}$ be the algebraic closure of the fraction filed of $\ZZ[\frac{1}{2}][Q_1,\cdots,Q_{m}]$. Set $Q_0=1$, we use the following cyclotomic polynomial $$h(x):=\sum_{\substack{0\leq t\leq \frac{d}{2}, \\ t\in \NN}}Q_{t}x^{d-2t}$$ to define the generic cyclotomic Sergeev algebra $\mhcn^{h}$  over ${\rm K}$  (or $\ZZ[\frac{1}{2}][Q_1,\cdots,Q_{m}]$).
\begin{theorem}\cite{SW}\label{rep}
The	generic cyclotomic Sergeev superalgebra $\mhcn^{h}$ is semisimple over ${\rm K}.$ If $d=2m$ is even, then the number of its simple modules is $|\mathscr{P}^{\mathsf{0},m}_n|$ and the number of its simple modules of {\em type }\texttt{M} is equal to $|\mathscr{MP}^{\mathsf{0}, m}_n|.$ If $d=2m+1$ is odd, then the number of its simple modules is $|\mathscr{P}^{\mathsf{s},m}_n|$ and the number of its simple modules of {\em type }\texttt{M} is equal to $|\mathscr{MP}^{\mathsf{s},m}_n|.$
	\end{theorem}
	
\begin{corollary}\label{semisimple2}
	 We have $${\text{dim}}_{\rm K}  {\rm Tr}(\mhcn^{h})_{\bar{0}}= {\text{dim}}_{\rm K}  {\rm Z}(\mhcn^{h})_{\bar{0}}=\begin{cases}|\mathscr{P}^{\mathsf{0},m}_n|, &\text{if $d=2m$ is even},\\
	 |\mathscr{P}^{\mathsf{s},m}_n|, &\text{if $d=2m+1$ is odd.}
	 	\end{cases}$$ and $${\text{dim}}_{\rm K}  {\rm SupTr}(\mhcn^{h})=\begin{cases}|\mathscr{MP}^{\mathsf{0},m}_n|, &\text{if $d=2m$ is even,}\\
	 	|\mathscr{MP}^{\mathsf{s},m}_n|, &\text{if $d=2m+1$ is odd.}
	 	\end{cases}$$
	\end{corollary}
	\begin{proof}
		This follows from Corollary \ref{semisimple1} and Theorem \ref{rep}.
	\end{proof}
\section{(Super)symmetrizing form on $\mhcn^{g}$}

In this section, we shall prove Theorem \ref{dengerate}. We recall the Frobenius form on $\mathfrak{H}_{n}^{g}$ given in \cite{K2} first. Rewrite the decomposition of $\mathfrak{H}_{n+1}^{g}$ in \eqref{decomposition} as
$$\mathfrak{H}_{n+1}^{g}=x_{n+1}^{d-1}\mathfrak{H}_{n}^{g} \oplus \bigoplus_{a=0}^{d-2}x_{n+1}^{a}\mathfrak{H}_{n}^{g}
                  \oplus \bigoplus_{a=0}^{d-1}x_{n+1}^{a}c_{n+1}\mathfrak{H}_{n}^{g} \oplus \mathfrak{H}_{n}^{g}s_{n}\mathfrak{H}_{n}^{g}.$$

Let $\theta_{n+1}: \mathfrak{H}_{n+1}^{g} \rightarrow \mathfrak{H}_{n}^{g}$ be the projection on to the $\mathfrak{H}_{n}^{g}$-coefficient of the first summand of this decomposition, that is, if $$y=\sum_{a=0}^{d-1} \left( x_{n+1}^{a}\sigma_{a}+x_{n+1}^{a}c_{n+1}\tau_{a} \right)+h_1s_nh_2
 \in \mathfrak{H}_{n+1}^{g},$$ where $\sigma_{a}, \tau_{a}, h_1, h_2\in \mathfrak{H}_{n}^{g}$, then $\theta_{n+1}(y):=\sigma_{d-1}$. We have the following.

\begin{lemma}(\cite[Lemma 15.6.2]{K2})
The $(\mathfrak{H}_{n}^{g}, \mathfrak{H}_{n}^{g})$-bisupermodule homomorphism $\theta_{n+1}$ is non-degenerate, i.e., $\ker \theta_{n+1}$ contains no
non-zero left ideals of $\mathfrak{H}_{n+1}^{g}.$
\end{lemma}

Hence we can define a Frobenius form on $\mathfrak{H}_{n}^{g}$ as follows:
$$t_{n, d}:=\theta_{1} \circ \theta_{2} \circ \cdots \circ \theta_{n}, \quad \mathfrak{H}_{n}^{g} \rightarrow{\rm R}.$$  We want to prove that the form $t_{n, d}$ is symmetric when level $d$ is odd and $t_{n, d}$ is super symmetric when level $d$ is even. To this end, we need the following Lemma.

We use $\iota_n: \mathfrak{H}_{n}^{g}\rightarrow \mathfrak{H}_{n+1}^{g}$ to denote the natural embedding.

\begin{lemma}\label{degeneratelem}
For $n \geq 1$ and any $x \in \mathfrak{H}_{n}^{g} \subseteq \mathfrak{H}_{n+1}^{g},$ we have $ \theta_{n+1}(s_{n}xs_{n})=\iota_{n-1}\circ\theta_{n}(x) $ in $\mathfrak{H}_{n}^{g}.$
\end{lemma}
\begin{proof}
Since $\theta_{n+1}$, $\theta_{n}$ and $\iota_{n-1}$ are all right $\mathfrak{H}_{n-1}^{g}$-linear,
using \eqref{free}, we can reduce $x$ to the following two cases.

\smallskip

{\it Case 1.} $x=x_{n}^{a}c_{n}^{b},$ for some $0\leq a <d,$ $b\in \Bbb Z_{2}.$

In this case, we deduce from \ref{Formulae} (i) that
$$s_{n}xs_{n}=s_{n} x_{n}^{a} c_{n}^{b} s_{n} = x_{n+1}^{a} c_{n+1}^{b}-(\ast),$$
where
$$(\ast)=\sum_{k=0}^{a-1}(x_{n}^{k}x_{n+1}^{a-1-k}+(-x_{n})^{k}x_{n+1}^{a-1-k})c_{n}^{b}s_{n}
\in \bigoplus_{a=0}^{d-2}x_{n+1}^{a}\mathfrak{H}_{n}^{g}
                  \oplus \bigoplus_{a=0}^{d-1}x_{n+1}^{a}c_{n+1}\mathfrak{H}_{n}^{g} \oplus \mathfrak{H}_{n}^{g}s_{n}\mathfrak{H}_{n}^{g}.$$
It follows that
$$\theta_{n+1}(s_{n}xs_{n})=\theta_{n+1}(x_{n+1}^{a} c_{n+1}^{b})=\delta_{(a, b), (d-1, 0)}=\iota_{n-1}\circ\theta_{n}(x_{n}^{a} c_{n}^{b})=\iota_{n-1}\circ\theta_{n}(x).$$

\smallskip

{\it Case 2.} $x=x_{j}^{a}c_{j}^{b}s_{j}\cdots s_{n-1}\in \mathfrak{H}_{n-1}^{g}s_{n-1}\mathfrak{H}_{n-1}^{g},$ for some $0\leq a <d,$ $b\in \Bbb Z_{2}$
and $1 \leq j < n.$

In this case, it is clearly that $\theta_{n}(x)=0.$ Since $j<n$, we have
$$
\begin{aligned}
s_{n}xs_{n} &= s_{n}x_{j}^{a}c_{j}^{b}s_{j}\cdots s_{n-1}s_{n} \nonumber \\
            &= x_{j}^{a}c_{j}^{b}s_{j} \cdots s_{n-2}s_{n}s_{n-1}s_{n} \nonumber\\
            &= x_{j}^{a}c_{j}^{b}s_{j} \cdots s_{n-2}s_{n-1}s_{n}s_{n-1}\in \mathfrak{H}_{n}^{g}s_{n}\mathfrak{H}_{n}^{g}, \nonumber
\end{aligned}
$$
from which we deduce $\theta_{n+1}(s_{n}xs_{n})=0=\theta_{n}(x).$

\end{proof}

{\bf Proof of Theorem \ref{dengerate}}:
We prove that the form $t_{n, d}$ is symmetric when level $d$ is odd and $t_{n, d}$ is super symmetric when level $d$ is even for any $d\in \Bbb N$. Since each $\theta_{i} $ is homogeneous of degree $0$, we have $t_{n, d}\bigl((\mathfrak{H}_{n}^{g})_{\overline{1}}\bigr)=0$. We only need to show the following: \begin{equation}\label{MainFormul}
t_{n, d}(xy)=(-1)^{(d-1)|x||y|}t_{n, d}(yx),\quad \text{for any homogeneous $x, y \in \mathfrak{H}_{n}^{g},\,d\in \Bbb N$.}
\end{equation}  We show \eqref{MainFormul} by induction on $n \in \Bbb N$.

For $n=0,$ $d \in \Bbb N,$ we denote $\mathfrak{H}_{0}^{g}=(\mathfrak{H}_{0}^{g})_{\overline{0}}={\rm R}$
and $t_{0, d}=\theta_{0}={\rm Id}_{{\rm R}}.$

For $n=1,$ $d \in \Bbb N,$ the algebra $\mathfrak{H}_{1}^{g}$ is generated by $x_{1}$ and $c_{1}.$ Recall the definition of $f,$ we deduce $\theta_{1}(x_{1}^{d})=-a_{d-1}=0,\,\theta_{1}(x_{1}^{d+1})=-a_{d-2}$
and
$$\theta_{1}(x_{1}^{d+k})=\sum_{\substack{0< t\leq \frac{d}{2}, \\ t\in \NN}}-a_{d-2t}\theta_{1}(x_{1}^{d+k-2t}), \quad k \geq 2.$$
Induction on $k\in \Bbb N,$ we have
$t_{1, d}(x_{1}^{d+k})=F_{k}(a_{d-2}, a_{d-4}, \cdots )$ for some polynomial $F_{k}$ who has no constant term on $a_{d-2}, a_{d-4}, \ldots,$ and \begin{equation}\label{vanish}
                                     F_{k}=0,\,\qquad \text{if $k$ is even.}
                                     \end{equation}
We shall prove the equation \eqref{MainFormul}. There are two cases. \smallskip

{\it Case 1.}  $0 \leq a, a' < d$ satisfy $a+a' < d$ and $b, b' \in \Bbb Z_{2}.$  We have
$$
\begin{aligned}
\theta_{1}(x_{1}^{a} c_{1}^{b} x_{1}^{a'} c_{1}^{b'}) &= (-1)^{ba'} \delta_{b, b'} \delta_{a+a', d-1}, \nonumber \\
\theta_{1}(x_{1}^{a'} c_{1}^{b'} x_{1}^{a} c_{1}^{b}) &= (-1)^{b'a} \delta_{b, b'} \delta_{a+a', d-1}, \nonumber
\end{aligned}
$$ and the above two equations are nonzero if and only if $b=b'$ and $a+a'=d-1$. When $d$ is odd, this implies the above two equations are nonzero only if  $a,a'$ have the same parity and $\theta_{1}(x_{1}^{a} c_{1}^{b} x_{1}^{a'} c_{1}^{b'})=\theta_{1}(x_{1}^{a'} c_{1}^{b'} x_{1}^{a} c_{1}^{b})=(-1)^{d-1}\theta_{1}(x_{1}^{a'} c_{1}^{b'} x_{1}^{a} c_{1}^{b}).$ When $d$ is even, we have that the above two equations are nonzero only if  $a,a'$ have different parities and $\theta_{1}(x_{1}^{a} c_{1}^{b} x_{1}^{a'} c_{1}^{b'})=-\theta_{1}(x_{1}^{a'} c_{1}^{b'} x_{1}^{a} c_{1}^{b})=(-1)^{d-1}\theta_{1}(x_{1}^{a'} c_{1}^{b'} x_{1}^{a} c_{1}^{b}).$

\smallskip

{\it Case 2.}  $0 \leq a, a' < d$ satisfy $d \leq a+a' \leq 2d-2,$ and $b, b' \in \Bbb Z_{2}.$ We have
$$
\begin{aligned}
\theta_{1}(x_{1}^{a} c_{1}^{b} x_{1}^{a'} c_{1}^{b'}) &= (-1)^{ba'} \delta_{b, b'}\theta_{1}(x_{1}^{a+a'})
                                                       =(-1)^{ba'} \delta_{b, b'}F_{a+a'-d}(a_{d-2}, a_{d-4}, \cdots ), \nonumber \\
\theta_{1}(x_{1}^{a'} c_{1}^{b'} x_{1}^{a} c_{1}^{b}) &= (-1)^{b'a} \delta_{b, b'}\theta_{1}(x_{1}^{a+a'})
                                                        =(-1)^{b'a} \delta_{b, b'}F_{a+a'-d}(a_{d-2}, a_{d-4}, \cdots ), \nonumber
\end{aligned}
$$ and the above two equations are nonzero  only if $b=b'$ and $a+a'-d$ is odd by \eqref{vanish}. When $d$ is odd, this implies the above two equations are nonzero  only if $b=b'$ and $a+a'$ is even, i.e. $a,a'$ have the same parity. It follows $\theta_{1}(x_{1}^{a} c_{1}^{b} x_{1}^{a'} c_{1}^{b'})=\theta_{1}(x_{1}^{a'} c_{1}^{b'} x_{1}^{a} c_{1}^{b})
                                                            =(-1)^{ba} \delta_{b, b'}F_{a+a'-d}(a_{d-2}, a_{d-4}, \cdots ).$
When $d$ is even, we have that the above two equations are nonzero  only if $b=b'$ and $a+a'$ is odd, i.e. $a,a'$ have different parities. We deduce
$$\theta_{1}(x_{1}^{a} c_{1}^{b} x_{1}^{a'} c_{1}^{b'})=(-1)^{b}\theta_{1}(x_{1}^{a'} c_{1}^{b'} x_{1}^{a} c_{1}^{b})=(-1)^{ba'} \delta_{b, b'}F_{a+a'-d}(a_{d-2}, a_{d-4},\ldots ).$$

Hence \eqref{MainFormul} holds for $n=1.$

For $n>1,$ assume \eqref{MainFormul} holds for $n$, we want to prove \eqref{MainFormul} holds for $n+1$. We first claim that \begin{equation}\label{Mainstep}
t_{n+1, d}(xy)=(-1)^{(d-1)|x||y|}t_{n+1, d}(yx),\,\,\text{$\forall y \in \mathfrak{H}_{n+1}^{g}$ and $x\in \{x_{1}, \ldots, x_{n+1}, s_{1}, \ldots, s_{n}, c_{1}, \ldots, c_{n+1}\}.$}
\end{equation}
We divide the proof of \eqref{Mainstep} into four steps.

$Step$ 1. For each generator $x\in \{x_{1}, \ldots, x_{n}, s_{1}, \ldots, s_{n-1}, c_{1}, \ldots, c_{n}\}\subseteq \mathfrak{H}_{n+1}^{g},$
we prove that $t_{n+1, d}(xy)=(-1)^{(d-1)|x||y|}t_{n+1, d}(yx)$ for any $y \in \mathfrak{H}_{n+1}^{g},$ $d\in \Bbb N.$

In fact, we have
$$
\begin{aligned}
 t_{n+1, d}(xy) &= t_{n, d}\left( \theta_{n+1}(xy) \right) \nonumber \\
                &= t_{n, d}\left( x\theta_{n+1}(y) \right) \nonumber \\
                &= (-1)^{(d-1)|x||y|}t_{n, d}\left(\theta_{n+1}(y)x \right) \nonumber \\
                &= (-1)^{(d-1)|x||y|}t_{n, d}\left(\theta_{n+1}(yx) \right) \nonumber \\
                &= (-1)^{(d-1)|x||y|}t_{n+1, d}(xy),
\end{aligned}
$$
where in the second and fourth equation we have used that $\theta_{n+1}$ is $\mathfrak{H}_{n}^{g}$-bilinearity, in the third equation we have used that $\theta_{n+1}$ is homogeneous of degree $0$ and in the last equation we have used induction hypothesis.

$Step$ 2. For generator $x=c_{n+1}\in \mathfrak{H}_{n+1}^{g},$
we prove that $t_{n+1, d}(xy)=(-1)^{d-1}t_{n+1, d}(yx)$ for any $y \in (\mathfrak{H}_{n+1}^{g})_{\overline{1}},$ $d\in \Bbb N.$

Using Corollary \ref{decomposition2}, $y$ can be written as the following form:
\begin{equation}\label{y}
y=\sum_{a=0}^{d-1} \left( x_{n+1}^{a}\sigma_{a}+x_{n+1}^{a}c_{n+1}\tau_{a} \right) \\
     +\sum_{a=0}^{d-1}\sum_{j=1}^{n}\left( x_{j}s_{j}\cdots s_{n}\mu_{a,j} + x_{j}c_{j}s_{j}\cdots s_{n}\nu_{a,j} \right),
\end{equation}
where $\sigma_{a}, \tau_{a}, \mu_{a,j}, \nu_{a,j}\in \mathfrak{H}_{n}^{g}.$ We have $$xy=\sum_{a=0}^{d-1} \left( c_{n+1}x_{n+1}^{a}\sigma_{a}+c_{n+1}x_{n+1}^{a}c_{n+1}\tau_{a} \right) \\
     +\sum_{a=0}^{d-1}\sum_{j=1}^{n}\left( c_{n+1}x_{j}s_{j}\cdots s_{n}\mu_{a,j} + c_{n+1}x_{j}c_{j}s_{j}\cdots s_{n}\nu_{a,j} \right).
$$ Since \begin{align*}
          c_{n+1}x_{n+1}^{a}\sigma_{a}&=(-1)^ax_{n+1}^{a}c_{n+1}\sigma_{a}\in x_{n+1}^{a}c_{n+1}\mathfrak{H}_{n}^{g};\\
          c_{n+1}x_{n+1}^{a}c_{n+1}\tau_{a}&=(-1)^a x_{n+1}^{a}\tau_{a}\in x_{n+1}^{a}\mathfrak{H}_{n}^{g};\\
           c_{n+1}x_{j}s_{j}\cdots s_{n}\mu_{a,j}&=x_{j}s_{j}\cdots s_{n} c_{n}\mu_{a,j}  \in \mathfrak{H}_{n}^{g} s_n\mathfrak{H}_{n}^{g}; \\
          c_{n+1}x_{j}c_{j}s_{j}\cdots s_{n}\nu_{a,j}&= x_{j}c_{j}s_{j}\cdots s_{n}c_n\nu_{a,j} \in \mathfrak{H}_{n}^{g} s_n\mathfrak{H}_{n}^{g},
        \end{align*}  we deduce that
$t_{n+1, d}(xy)=t_{n, d}(\theta_n(xy))=(-1)^{d-1}t_{n,d}(\tau_{d-1})$.

Similarly, we have $$yx=\sum_{a=0}^{d-1} \left( x_{n+1}^{a}\sigma_{a}c_{n+1}+x_{n+1}^{a}c_{n+1}\tau_{a}c_{n+1} \right) \\
     +\sum_{a=0}^{d-1}\sum_{j=1}^{n}\left(x_{j}s_{j}\cdots s_{n}\mu_{a,j} c_{n+1} + x_{j}c_{j}s_{j}\cdots s_{n}\nu_{a,j} c_{n+1} \right).
$$ Note that $|y|=1$, which implies $|\tau_{a}|=|\nu_{a,j}|=\overline{0}$ and $|\sigma_{a}|=|\mu_{a,j}|=\overline{1}$. Hence \begin{align*}
         x_{n+1}^{a}\sigma_{a} c_{n+1}&=-x_{n+1}^{a}c_{n+1}\sigma_{a}\in x_{n+1}^{a}c_{n+1}\mathfrak{H}_{n}^{g};\\
         x_{n+1}^{a}c_{n+1}\tau_{a} c_{n+1}&= x_{n+1}^{a}\tau_{a}\in x_{n+1}^{a}\mathfrak{H}_{n}^{g};\\
           x_{j}s_{j}\cdots s_{n}\mu_{a,j}c_{n+1}&=-x_{j}c_{j}s_{j}\cdots s_{n} \mu_{a,j}  \in \mathfrak{H}_{n}^{g} s_n\mathfrak{H}_{n}^{g}; \\
         x_{j}c_{j}s_{j}\cdots s_{n}\nu_{a,j} c_{n+1}&= x_{j} s_{j}\cdots s_{n}\nu_{a,j} \in \mathfrak{H}_{n}^{g} s_n\mathfrak{H}_{n}^{g},
        \end{align*} and we deduce $t_{n+1, d}(yx)=t_{n, d}(\theta_n(yx))=t_{n,d}(\tau_{d-1})$. Combining with the result in last paragraph, we have
$$t_{n+1, d}(xy)=(-1)^{d-1}t_{n, d}(\tau_{d-1})=(-1)^{d-1}t_{n+1, d}(yx).$$

$Step$ 3. For generator $x=x_{n+1}\in \mathfrak{H}_{n+1}^{g},$
we prove that $t_{n+1, d}(xy)=t_{n+1, d}(yx)$ for any $y \in (\mathfrak{H}_{n+1}^{g})_{\overline{0}},$ $d\in \Bbb N.$

Again, using Corollary \ref{decomposition2},  $y$ can be written as the following form:
\begin{equation}
y=\sum_{a=0}^{d-1} \left( x_{n+1}^{a}\sigma_{a}+x_{n+1}^{a}c_{n+1}\tau_{a} \right) \\
     +\sum_{a=0}^{d-1}\sum_{j=1}^{n}\left( x_{j}s_{j}\cdots s_{n}\mu_{a,j} + x_{j}c_{j}s_{j}\cdots s_{n}\nu_{a,j} \right),
\end{equation}
where $\sigma_{a}, \tau_{a}, \mu_{a,j}, \nu_{a,j}\in \mathfrak{H}_{n}^{g}.$ Then we have
$$
\begin{aligned}
 xy &=x_{n+1}^{d}\sigma_{d-1}+x_{n+1}^{d}c_{n+1}\tau_{d-1}+x_{n+1}^{d-1}\sigma_{d-2}+x_{n+1}^{d-1}c_{n+1}\tau_{d-2}
                     +\sum_{a=0}^{d-3} \left( x_{n+1}^{a+1}\sigma_{a}+x_{n+1}^{a+1}c_{n+1}\tau_{a} \right) \\
    &+ \sum_{a=0}^{d-1}\sum_{j=1}^{n}\left( x_{j}s_{j}\cdots s_{n-1} (s_{n}x_{n}+1+c_{n}c_{n+1})\mu_{a,j}
                  + x_{j}c_{j}s_{j}\cdots s_{n-1} (s_{n}x_{n}+1+c_{n}c_{n+1}) \nu_{a,j} \right) \nonumber \\
    &\in x_{n+1}^{d}\sigma_{d-1}+x_{n+1}^{d}c_{n+1}\tau_{d-1}+x_{n+1}^{d-1}\sigma_{d-2}+\bigoplus_{a=0}^{d-2}x_{n+1}^{a}\mathfrak{H}_{n}^{g}
                  \oplus \bigoplus_{a=0}^{d-1}x_{n+1}^{a}c_{n+1}\mathfrak{H}_{n}^{g} \oplus \mathfrak{H}_{n}^{g}s_{n}\mathfrak{H}_{n}^{g}.
\end{aligned}
$$ Using Lemma \ref{Formulae} (ii), we can deduce  $$0=s_n\cdots s_1f(x_1)s_1\cdots s_n \in x_{n+1}^{d} +\mathfrak{H}_{n}^{g}s_{n}\mathfrak{H}_{n}^{g} \oplus \bigoplus_{0 \leq k \leq d-2, b\in \Bbb Z_{2}}x_{n+1}^{k}c_{n+1}^{b}\mathfrak{H}_{n}^{g},$$ hence
\begin{align}
x_{n+1}^{d}\sigma_{d-1} &\in \mathfrak{H}_{n}^{g}s_{n}\mathfrak{H}_{n}^{g} \oplus \bigoplus_{0 \leq k \leq d-2, b\in \Bbb Z_{2}}x_{n+1}^{k}c_{n+1}^{b}\mathfrak{H}_{n}^{g},\label{little inclusion1}\\
x_{n+1}^{d}c_{n+1}\tau_{d-1}& \in \mathfrak{H}_{n}^{g}s_{n}\mathfrak{H}_{n}^{g} \oplus \bigoplus_{0 \leq k \leq d-2, b\in \Bbb Z_{2}}x_{n+1}^{k}c_{n+1}^{b}\mathfrak{H}_{n}^{g},\label{little inclusion2}
\end{align}

which implies $t_{n+1, d}(xy)=t_{n, d}(\theta_{n+1}(xy))=t_{n, d}(\sigma_{d-2}).$

Similarly, we have
\begin{align*}
 y x &=\sum_{a=0}^{d-1} \left( x_{n+1}^{a}\sigma_{a}x_{n+1}+x_{n+1}^{a}c_{n+1}\tau_{a} x_{n+1} \right) \\
    &+ \sum_{a=0}^{d-1}\sum_{j=1}^{n}\left( x_{j}s_{j}\cdots s_{n-1} s_n\mu_{a,j}x_{n+1}
                  + x_{j}c_{j}s_{j}\cdots s_{n-1} s_n\nu_{a,j}x_{n+1} \right) \nonumber \\
    &=\sum_{a=0}^{d-1} \left( x_{n+1}^{a+1}\sigma_{a}-x_{n+1}^{a+1}c_{n+1}\tau_{a}  \right) \\
    &+ \sum_{a=0}^{d-1}\sum_{j=1}^{n}\left( x_{j}s_{j}\cdots s_{n-1} (x_ns_n+1-c_nc_{n+1})\mu_{a,j}
                  + x_{j}c_{j}s_{j}\cdots s_{n-1} (x_ns_n+1-c_nc_{n+1})\nu_{a,j} \right) \nonumber \\
    &\in x_{n+1}^{d}\sigma_{d-1}+x_{n+1}^{d}c_{n+1}\tau_{d-1}+x_{n+1}^{d-1}\sigma_{d-2}+\bigoplus_{a=0}^{d-2}x_{n+1}^{a}\mathfrak{H}_{n}^{g}
                  \oplus \bigoplus_{a=0}^{d-1}x_{n+1}^{a}c_{n+1}\mathfrak{H}_{n}^{g} \oplus \mathfrak{H}_{n}^{g}s_{n}\mathfrak{H}_{n}^{g}\\
    &\in x_{n+1}^{d-1}\sigma_{d-2}+\bigoplus_{a=0}^{d-2}x_{n+1}^{a}\mathfrak{H}_{n}^{g}
                  \oplus \bigoplus_{a=0}^{d-1}x_{n+1}^{a}c_{n+1}\mathfrak{H}_{n}^{g} \oplus \mathfrak{H}_{n}^{g}s_{n}\mathfrak{H}_{n}^{g},
\end{align*}
where we have used \eqref{little inclusion1} and \eqref{little inclusion2} in the last inclusion. Hence we deduce $t_{n+1, d}(xy)= t_{n, d}(\sigma_{d-2})=t_{n+1, d}(yx).$

$Step$ 4. For generator $x=s_{n}\in \mathfrak{H}_{n+1}^{g},$
we prove that $t_{n+1, d}(xy)=t_{n+1, d}(yx)$ for any $y \in (\mathfrak{H}_{n+1}^{g})_{\overline{0}},$ $d\in \Bbb N.$

Due to the decomposition (\ref{decomposition}), we can reduce $y$ to the following two cases.

(i) If $y=x_{n+1}^{a}c_{n+1}^{b}\tilde{y},$ for some $0 \leq a < d,$ $b\in \Bbb Z_{2}$ and $\tilde{y} \in \left(\mathfrak{H}_{n}^{g}\right)_{b},$ then

\begin{align*}
 xy &= s_{n}x_{n+1}^{a}c_{n+1}^{b}\tilde{y} \nonumber \\
    &= \left( x_{n}^{a}s_{n}+\sum_{k=0}^{a-1}(x_{n}^{k}x_{n+1}^{a-1-k}-x_{n}^{k}(-x_{n+1})^{a-1-k}c_{n}c_{n+1})\right)c_{n+1}^{b}\tilde{y} \nonumber \\
    &= x_{n}^{a}c_{n}^{b}s_{n}\tilde{y}
    +\sum_{k=0}^{a-1}(x_{n+1}^{a-1-k}c_{n+1}^{b}x_{n}^{k}\tilde{y}-(-1)^{a+b}x_{n+1}^{a-1-k}c_{n+1}^{b+1}c_{n}x_{n}^{k}\tilde{y})
    \nonumber \\
    &\in \mathfrak{H}_{n}^{g}s_{n}\mathfrak{H}_{n}^{g} \oplus \bigoplus_{0 \leq k \leq d-2, b\in \Bbb Z_{2}}x_{n+1}^{k}c_{n+1}^{b}\mathfrak{H}_{n}^{g}.
\end{align*}

On the other hand, we have $c_{n+1}^{b}\widetilde{y}=(-1)^{b}\widetilde{y}c_{n+1}^{b},$
then
$$
\begin{aligned}
 yx &= x_{n+1}^{a}c_{n+1}^{b}\tilde{y} s_{n}
    = (-1)^{b}\widetilde{y}x_{n+1}^{a}c_{n+1}^{b}s_{n}
    = (-1)^{b}\widetilde{y}x_{n+1}^{a}s_{n} c_{n}^{b}\nonumber \\
    &= (-1)^{b}\widetilde{y} \left( s_{n} x_{n}^{a}+\sum_{k=0}^{a-1}(x_{n}^{k}x_{n+1}^{a-1-k}
    +(-x_{n})^{k}x_{n+1}^{a-1-k}c_{n}c_{n+1})\right)c_{n}^{b} \nonumber \\
    &= (-1)^{b}\widetilde{y}s_{n} x_{n}^{a}c_{n}^{b}
    +\sum_{k=0}^{a-1}((-1)^{b}x_{n+1}^{a-1-k}\widetilde{y}x_{n}^{k}c_{n}^{b}
    -(-1)^{b}\widetilde{y}(-x_{n})^{k}x_{n+1}^{a-1-k}c_{n+1}c_{n}^{b+1}) \nonumber \\
    &=  (-1)^{b}\widetilde{y}s_{n} x_{n}^{a}c_{n}^{b}
    +\sum_{k=0}^{a-1}((-1)^{b}x_{n+1}^{a-1-k}\widetilde{y}x_{n}^{k}c_{n}^{b}
    -(-1)^{b+1}x_{n+1}^{a-1-k}c_{n+1}\widetilde{y}(-x_{n})^{k}c_{n}^{b+1}) \nonumber \\
    &\in \mathfrak{H}_{n}^{g}s_{n}\mathfrak{H}_{n}^{g} \oplus \bigoplus_{0 \leq k \leq d-2, b\in \Bbb Z_{2}}x_{n+1}^{k}c_{n+1}^{b}\mathfrak{H}_{n}^{g}.
\end{aligned}
$$
In the case, we have $t_{n+1, d}(xy)=t_{n+1, d}(yx)=0.$

(ii) If $y=y's_{n}y'' \in \mathfrak{H}_{n}^{g}s_{n}\mathfrak{H}_{n}^{g},$ for some $y', y'' \in \mathfrak{H}_{n}^{g},$ then
$$
\begin{aligned}
 t_{n+1, d}(xy) &= t_{n, d}\left( \theta_{n+1}(s_{n}y's_{n}y'') \right) \nonumber \\
                &= t_{n, d}\left( \theta_{n+1}(s_{n}y's_{n})y'' \right) \nonumber \\
                &= t_{n, d}\left( \iota_{n-1}\circ \theta_{n}(y')y'' \right) \nonumber \\
                &= t_{n-1, d}\left(\theta_{n}\left(\iota_{n-1}\circ\theta_{n}(y')y'' \right)  \right) \nonumber \\
                &= t_{n-1, d}\left(\theta_{n}(y')\theta_{n}(y'')  \right),
\end{aligned}
$$
where in the first and the fourth equalities we have used definitions of $t_{n+1, d}$ and $t_{n, d},$ respectively, in the second and the last equalities we have used $\mathfrak{H}_{n}^{g}$-bilinearity of $\theta_{n+1}$ and $\mathfrak{H}_{n-1}^{g}$-bilinearity of $\theta_{n},$ respectively, and in the third equality we have used Lemma \ref{degeneratelem}.

Similarly, we have
$$
\begin{aligned}
 t_{n+1, d}(yx) &= t_{n, d}\left( \theta_{n+1}(y's_{n}y''s_{n}) \right) \nonumber \\
                &= t_{n, d}\left( y' \theta_{n+1}(s_{n}y''s_{n}) \right) \nonumber \\
                &= t_{n, d}\left( y' \iota_{n-1}\circ\theta_{n}(y'') \right) \nonumber \\
                &= t_{n-1, d}\left(\theta_{n}\left(y'\iota_{n-1}\circ\theta_{n}(y'') \right)  \right) \nonumber \\
                &= t_{n-1, d}\left(\theta_{n}(y')\theta_{n}(y'')  \right).
\end{aligned}
$$
This completes the proof of \eqref{Mainstep}.

Finally, we claim that  for any $d \in \Bbb N,$
if $t_{n+1, d}(xy)=(-1)^{(d-1)|x||y|}t_{n+1, d}(yx)$ and $t_{n+1, d}(xz)=(-1)^{(d-1)|x||z|}t_{n+1, d}(zx)$
for all $x \in\mathfrak{H}_{n+1}^{g},$ then $t_{n+1, d}(xyz)=(-1)^{(d-1)|x||yz|}t_{n+1, d}(yzx)$ for all $x \in\mathfrak{H}_{n+1}^{g}.$

In fact,
$$
\begin{aligned}
 t_{n+1, d}(xyz) &= (-1)^{(d-1)|z||xy|}t_{n+1, d}(z xy) \nonumber \\
                 &= (-1)^{(d-1)(|z||x|+|z||y|)}(-1)^{(d-1)|y||zx|}t_{n+1, d}(y zx)\nonumber \\
                 &= (-1)^{(d-1)(|z||x|+|z||y|+|y||z|+|y||x|)}t_{n+1, d}(y zx) \nonumber \\
                 &= (-1)^{(d-1)|x||yz|}t_{n+1, d}(yzx).
\end{aligned}
$$ This combining with \eqref{Mainstep} completes the proof of \eqref{MainFormul}.
\qed

\begin{theorem}
Let $n, d \in \Bbb N.$ The (super)symmetric form $t_{n, d}$ has the following explicit formula
$$
t_{n, d}(x^{\alpha}c_I w)=
\begin{cases}
1,& \text{{\rm if} $\alpha_{1}= \cdots = \alpha_{n}=d-1, I=\emptyset, w=1$};\\
0, & \text{{\rm otherwise}},
\end{cases}
$$
for any $0 \leq \alpha_{1}, \ldots, \alpha_{n} < d,$ $I \subseteq [n],$ $w \in \mathfrak{S}_{n}.$
\end{theorem}
\begin{proof}
We use induction on $n.$ For $n=1$, this is trival. Now we assume the statement is true for $n-1 \geq 1$.
We write $c_I=c_{1}^{\beta_{1}} \cdots c_{n}^{\beta_{n}}$ for some $\beta_1,\ldots,\beta_n \in \mathbb{Z}_2.$

(i) If $w\notin\mathfrak{S}_{n-1}$, then $w=s_{j} \cdots s_{n-1}w',$ for some $1 \leq j < n,$ $w' \in\mathfrak{S}_{n-1}.$ We have
$$
\begin{aligned}
t_{n, d}(x^{\alpha}c_I w)
 &= t_{n, d}(x_{1}^{\alpha_{1}} \cdots x_{n}^{\alpha_{n}} c_{1}^{\beta_{1}} \cdots c_{n}^{\beta_{n}} w) \nonumber \\
 &= t_{n-1, d}(\theta_{n} (x_{1}^{\alpha_{1}} \cdots x_{n}^{\alpha_{n}} c_{1}^{\beta_{1}} \cdots c_{n}^{\beta_{n}} s_{j} \cdots s_{n-1}w'))\nonumber \\
 &= t_{n-1, d}(x_{1}^{\alpha_{1}} \cdots x_{n-1}^{\alpha_{n-1}} c_{1}^{\beta_{1}} \cdots c_{n-1}^{\beta_{n-1}}s_{j} \cdots s_{n-2} \theta_{n}(x_{n}^{\alpha_{n}}c_{n}^{\beta_{n}}s_{n-1}) w') \nonumber \\
 &= t_{n-1, d}(x_{1}^{\alpha_{1}} \cdots x_{n-1}^{\alpha_{n-1}} c_{1}^{\beta_{1}} \cdots c_{n-1}^{\beta_{n-1}}s_{j} \cdots s_{n-2} \theta_{n}(x_{n}^{\alpha_{n}}s_{n-1}) c_{n-1}^{\beta_{n}}w'),\nonumber
\end{aligned}
$$ by $\mathfrak{H}_{n}^{g}$-bilinearity of $\theta_{n}$
Since
$$x_{n}^{\alpha_{n}}s_{n-1}=s_{n-1}x_{n-1}^{\alpha_{n}}
                     +\sum_{k=0}^{\alpha_{n}-1}(x_{n-1}^{k}x_{n}^{\alpha_{n}-1-k}+(-x_{n-1})^{k}x_{n}^{\alpha_{n}-1-k}c_{n-1}c_{n})$$
$$\in \mathfrak{H}_{n-1}^{g}s_{n-1}\mathfrak{H}_{n-1}^{g} \oplus \bigoplus_{0 \leq k \leq d-2, b\in \Bbb Z_{2}}x_{n}^{k}c_{n}^{b}\mathfrak{H}_{n-1}^{g},$$
we have $\theta_{n}(x_{n}^{\alpha_{n}}s_{n-1})=0,$ and $t_{n, d}(x^{\alpha}c_I w)=0$.

(ii) If  $w\in\mathfrak{S}_{n-1}$ then
$$
\begin{aligned}
t_{n, d}(x^{\alpha}c_I w)
 &= t_{n, d}(x_{1}^{\alpha_{1}} \cdots x_{n}^{\alpha_{n}} c_{1}^{\beta_{1}} \cdots c_{n}^{\beta_{n}} w) \nonumber \\
 &= t_{n-1, d}(\theta_{n} (x_{1}^{\alpha_{1}} \cdots x_{n}^{\alpha_{n}} c_{1}^{\beta_{1}} \cdots c_{n}^{\beta_{n}} w))\nonumber \\
 &= t_{n-1, d}(x_{1}^{\alpha_{1}} \cdots x_{n-1}^{\alpha_{n-1}} c_{1}^{\beta_{1}} \cdots c_{n-1}^{\beta_{n-1}} \theta_{n}(x_{n}^{\alpha_{n}}c_{n}^{\beta_{n}}) w) \nonumber \\
 &= \delta_{(\alpha_{n}, \beta_{n}), (d-1, 0)} t_{n-1, d}(x_{1}^{\alpha_{1}} \cdots x_{n-1}^{\alpha_{n-1}} c_{1}^{\beta_{1}} \cdots c_{n-1}^{\beta_{n-1}} w) \nonumber \\
 &= \delta_{(\alpha, I, w), (\underline{d-1}, \emptyset, 1)}
\end{aligned}
$$
where in the second equality, we have used $\mathfrak{H}_{n-1}^{g}$-bilinearity of $\theta_{n}$ and in the last equality, we have used induction hypothesis. This completes the proof of Theorem.
\end{proof}

The following tiny example explains that one can not find a symmetrizing form on $\mhcn^{g}$ when $d$ is even. So the supersymmetricity of $\mhcn^{g}$
in even level seems to be essential.
\begin{example}\label{counterexamp}
Let $n=1$ and level $d=2.$
We set $g(x_1)=x_1^{2},$ then $\mhcn^{g}$ has a basis $\{1,x_1,c_1,c_1x_1\}.$
Suppose that there is a symmetrizing form ${\rm tr}:\mhcn^{g} \rightarrow {\rm R}$ on $\mhcn^{g}$ satisfying ${\rm tr}((\mhcn^{g})_{\bar{1}})=0.$
Then we have
\begin{align}
{\rm tr}(x_1)={\rm tr}(c_1^2 x_1)={\rm tr}(c_1x_1c_1)=-{\rm tr}(x_1),\nonumber
\end{align}
where in the  second equation we have used the symmetricity of ${\rm tr}$ and in the third equation we have used relations \eqref{clifford} and \eqref{xc}. Since $2\neq 0$, we deduce ${\rm tr}(x_1)=0$.
It follows that
$${\rm tr}(1 \cdot x_1)={\rm tr}(x_1\cdot x_1)={\rm tr}(c_1 \cdot x_1)={\rm tr}(c_1x_1\cdot x_1)=0.$$ This implies that in the Gram matrix of ${\rm tr}$, there is one row to be a zero.  Hence ${\rm tr}$ is degenerated.
\end{example}

\section{A basis of ${\rm Tr}(\mhcn^{g})_{\overline{0}}$}
In this section, we assume ${\rm R}$ is an integral domain with $2$ invertible. We will construct a basis of the ${\rm Tr}(\mhcn^{g})_{\overline{0}}$. Recall that we have fixed $d\geq 0,\,m:=\lfloor \frac{d}{2}\rfloor$ and $$g(x)=\sum_{\substack{0\leq t\leq \frac{d}{2}, \\ t\in \NN}}a_{d-2t}x^{d-2t}$$ such that $a_d=1.$ Let's first recall some basic combinatorial concepts on complex reflection group $G(d,1,n)$.
\subsection{Some basics on complex reflection group $G(d,1,n)$} \label{minimal element}
Recall the definition of complex reflection group $W_{d,n}$ of type $G(d,1,n)$. Using the defining relations for $W_{d,n}$, it is easy to check that for any $a,\,b\in\ZZ^{\geq 0}$, \begin{equation}\label{brl0} s_0^as_1s_0^bs_1=s_1s_0^bs_1s_0^a .
\end{equation}
We call the relation (\ref{brl0}) the weak braid relations for $W_{d,n}$.

\begin{definition}
For any two words on $S$, say $s_{i_1}\cdots s_{i_k}$ and $s_{i'_1}\cdots s_{i'_t}$, where $s_{i_j}, s_{i'_j}\in S$, we say they are {weakly braid-equivalent} if we can use a sequence of braid relations together with the additional relation (\ref{brl0}) to transform from one into another.
\end{definition}


Given $w\in W_{d,n}$, a word $s_{i_1}\cdots s_{i_k}$ is called an expression of $w$ if $s_{i_j}\in S, \forall\,1\leq i_j\leq k$, and $w=s_{i_1}\cdots s_{i_k}$. If $s_{i_1}\cdots s_{i_k}$ is an expression of $w$ with $k$ minimal, then we call it a reduced expression of $w$. In this case, following \cite{BM}, we define $\ell(w):=k$. For each conjugacy class $C$ of $W_{r,n}$, we use $C_{\min}$ to denote the set of minimal length elements in $C$.

\begin{definition} An element $w\in W_{d,n}$ is called minimal if $w\in C_{\min}$ for some conjugacy class $C$ of $W_{d,n}$.
\end{definition}

We need further combinatorial notions to describe some special minimal length elements in conjugacy classes of $W_{d,n}$ which will be used in the next subsection.

\begin{definition}\text{(\cite[Definition 3.13]{HS})} A composition $\lam=(\lam_1,\cdots,\lam_k)$ of $n$ is called an opposite partition if $\lam_1\leq\lam_2\leq\cdots\leq\lam_k$. We use $\OP_{n,-}$ to denote the set of opposite partitions of $n$. A color datum associated with an opposite partition $\lam=(\lam_1,\cdots,\lam_k)\in\OP_{n,-}$ is a function $c:\{1,2,\cdots,k\} \to \{1,\cdots,d-1\}$ such that $c(i)\geq c(i+1)$ whenever $\lam_i=\lam_{i+1}$.
\end{definition}

\begin{definition}\text{(\cite[Definition 3.14]{HS})}\label{colorsemi} If $\lam$ is an opposite partition of $s$ with a color data $\{c(i)|1\leq i\leq\ell(\lam)\}$, $\mu$ is a composition of $n-s$, then we call the bicomposition $(\lam,\mu)$ a colored semi-bicomposition of $n$. We use
$\CC_{n}$ to denote the set of colored semi-bicomposition of $n$.
If $(\lam,\mu)$ is a colored semi-bicomposition of $n$ and $\mu$ is a partition, then we say $(\lam,\mu)$ is a colored semi-bipartition. We use $\CP_n$ to denote the set of colored semi-bipartitions of $n$.
\end{definition}

For each $0\leq k\leq n-1,\, l\in\ZZ^{\geq 1}$, we define
$$s'_{k,l}:=s_ks_{k-1}\cdots s_1 s^l_0s_1\cdots s_{k-1}s_k.$$ Let $\lambda=(\lambda_1,\cdots,\lambda_k)$ be a composition of $n$. We set $r_1:=0$, $r_{k+1}:=n$, and \begin{equation}\label{ri}
r_i:=\lam_1+\lam_2+\cdots+\lam_{i-1},\quad\forall\,2\leq i\leq k .
                                                         \end{equation}

Let $J:=\{0,1,\cdots,d-1\}$ and $\epsilon=(\epsilon_1,\cdots,\epsilon_k)\in J^k$. For each $1\leq i\leq k$, we define \begin{equation}\label{wlam}
w_{\lambda,\epsilon,i}:=\begin{cases}s'_{r_i,\epsilon_i}s_{r_i+1}s_{r_i+2}\cdots s_{r_{i+1}-1}, &\text{if $\epsilon_i\neq 0$;}\\
s_{r_i+1}s_{r_i+2}\cdots s_{r_{i+1}-1}, &\text{if $\epsilon_i=0$,}
\end{cases},\quad w_{\lambda,\epsilon}=\prod_{i=1}^k w_{\lambda,\epsilon,i}.
\end{equation}

For each colored semi-bicomposition ${\alpha}=(\lam,\mu)$, where $\lam=(\lam_1,\cdots,\lam_k)$ and  $\mu=(\mu_1,\cdots,\mu_l)$,  we associate it with a composition \begin{equation}\label{identification}\overline{{\alpha}}=(\overline{\alpha}_{1}, \ldots, \overline{\alpha}_{k+l}):=(\lam_1,\cdots,\lam_k,\mu_1,\cdots,\mu_l)
\end{equation} of $n$ and a sequence \begin{equation}\label{color}\epsilon=(c(1),\cdots,c(k),\underbrace{0,\cdots,0}_{\text{$l$ copies}})\in J^{k+l}.
\end{equation} We define \begin{equation}\label{walpha}
w_{\alpha}:=w_{\overline{\alpha},\epsilon}.
\end{equation}

\begin{lemma}\cite[Theorem 3.25]{HS}
Let $C$ be any conjugacy class of $W$ and $C_{\min}$ be the set of minimal length elements in $C$. Then \begin{enumerate}
\item there exists a unique $\beta_C\in\CP_n$ such that $w_{\beta_C}\in C$. Moreover, $w_{\beta_C}\in C_{\min}$;
\item for any $\alpha\in\CC_n$, $w_\alpha$ is a minimal length element in its conjugacy class.
\end{enumerate}
\end{lemma}

\subsection{A new presentation of $\mhcn^{g}$}
We define $(\mhcn^{g})'$ to be the superalgebra generated by even generators
$s_0,s_1,\ldots,s_{n-1}$ and odd generators
$c_1,\ldots,c_n$ subject to the following relations
\begin{align}
s_i^2=1, &\qquad f(s_0)=0, \label{square}\\
s_is_j =s_js_i, \quad
s_is_{i+1}s_i&=s_{i+1}s_is_{i+1}, \quad|i-j|>1, i\neq 0, 0\leq i,j < n, \label{braid'}\\
s_1s_0s_1s_0-s_0s_1s_0s_1&=s_0(1+c_1c_2)s_1-(1+c_1c_2)s_1s_0, \label{weakly braid relation'}\\
c_i^2=1,c_ic_j&=-c_jc_i, \quad 1\leq i\neq j\leq n,\\
s_ic_i=c_{i+1}s_i, s_ic_{i+1}&=c_is_i,s_ic_j=c_js_i, \quad j\neq i, i+1,\label{sc1} \\
s_0c_1=-c_1s_0, s_0c_j&=c_js_0, \label{sc2}\quad 1\leq j\leq n.
\end{align}

\begin{lemma}
  There is a surjective algebraic homomorphism $\Psi: (\mhcn^{g})'\rightarrow \mhcn^{g}$ such that $\Psi(s_0)=x_1,\, \Psi(s_i)=s_i,\,\forall 1\leq i\leq n-1,\,\Psi(c_j)=c_j,\forall 1\leq j\leq n$.
\end{lemma}

\begin{proof}
  To show $\Psi$ is an algebraic homomorphism, we only need to check the relation \eqref{weakly braid relation'} in $\mhcn^{g}$. Actually, by \eqref{px1}, we have $$s_1x_1s_1=x_2-(1+c_1c_2)s_1,
  $$ which implies \begin{align*}
    s_1x_1s_1x_1-x_1s_1x_1s_1&=x_1x_2-(1+c_1c_2)s_1x_1-\biggl(x_1x_2-x_1(1+c_1c_2)s_1\biggr)\\
    &=x_1(1+c_1c_2)s_1-(1+c_1c_2)s_1x_1.
  \end{align*} Hence it is an algebraic homomorphism. By  \eqref{px1} again, we see the generating set of $\mhcn^{g}$ belongs to the image. This implies $\Psi$ is surjective.
\end{proof}

We will show that $\Psi$ is an isomorphism.

\begin{lemma}
  The following equlities hold in $(\mhcn^{g})'$.\begin{equation}\label{weakly braid relation''}
  s_1s_0^as_1s_0^b-s_0^bs_1s_0^as_1=\sum_{i=1}^b\biggl(s_0^{a+b-i}(1+(-1)^{a-1}c_1c_2)s_1s_0^{i-1}
  -s_0^{i-1}(1+(-1)^{a-1}c_1c_2)s_1s_0^{a+b-i}\biggr),
 \end{equation} where $a,b\in \NN.$
\end{lemma}
\begin{proof}
  We first prove \eqref{weakly braid relation''} when $b=1$. We use induction on $a$. For $a=1$, this follows from defining relation \eqref{weakly braid relation'}. Suppose this is true for $a\geq 1$, i.e. $$
 s_1s_0^as_1s_0-s_0s_1s_0^as_1=\biggl(s_0^{a}(1+(-1)^{a-1}c_1c_2)s_1
  -(1+(-1)^{a-1}c_1c_2)s_1s_0^{a}\biggr).$$ Multiplying by $s_1s_0s_1$, on the left-hand of both sides, we get that \begin{align*}
 &\quad s_1s_0^{a+1}s_1s_0\\
 &=s_1s_0s_1s_0s_1s_0^as_1+\biggl(s_1s_0s_1s_0^{a}(1+(-1)^{a-1}c_1c_2)s_1
  -s_1s_0s_1(1+(-1)^{a-1}c_1c_2)s_1s_0^{a}\biggr)\\
  &=\biggl(s_0s_1s_0s_1+s_0(1+c_1c_2)s_1-(1+c_1c_2)s_1s_0)\biggr)s_1s_0^as_1\\
  &\qquad\qquad +\biggl(s_1s_0s_1s_0^{a}(1+(-1)^{a-1}c_1c_2)s_1
  -s_1s_0s_1(1+(-1)^{a-1}c_1c_2)s_1s_0^{a}\biggr)\\
  &=\biggl(s_0s_1s_0^{a+1}s_1+s_0(1+c_1c_2)s_0^as_1-(1+c_1c_2)s_1s_0s_1s_0^as_1)\biggr)\\
  &\qquad\qquad +\biggl(s_1s_0s_1s_0^{a}(1+(-1)^{a-1}c_1c_2)s_1
  -s_1s_0s_1(1+(-1)^{a-1}c_1c_2)s_1s_0^{a}\biggr)\\
   &=\biggl(s_0s_1s_0^{a+1}s_1+s_0^{a+1}(1+(-1)^ac_1c_2)s_1-(1+c_1c_2)s_1s_0s_1s_0^as_1)\biggr)\\
  &\qquad\qquad +\biggl((1+c_1c_2)s_1s_0s_1s_0^{a}s_1
  -(1+(-1)^{a}c_1c_2)s_1s_0^{a+1}\biggr)\\
  &=s_0s_1s_0^{a+1}s_1+s_0^{a+1}(1+(-1)^ac_1c_2)s_1 -(1+(-1)^{a}c_1c_2)s_1s_0^{a+1},
\end{align*}
where in the second equation we have used relation \eqref{weakly braid relation''} and in the fourth equation we have used relation \eqref{sc1}, \eqref{sc2}. This completes the proof for $b=1$. Now we show \eqref{weakly braid relation''} by induction on $b$. Suppose this is true for $b$ and all $a\in \NN$. Then for $\forall a\in \NN$, we have \begin{align*}
&\quad s_1s_0^as_1s_0^{b+1}\\
&=s_0^bs_1s_0^as_1s_0+\sum_{i=1}^b\biggl(s_0^{a+b-i}(1+(-1)^{a-1}c_1c_2)s_1s_0^{i}
  -s_0^{i-1}(1+(-1)^{a-1}c_1c_2)s_1s_0^{a+b-i+1}\biggr)\\
  &=s_0^b\biggl(s_0s_1s_0^as_1+\biggl(s_0^{a}(1+(-1)^{a-1}c_1c_2)s_1
  -(1+(-1)^{a-1}c_1c_2)s_1s_0^{a}\biggr)\biggr)\\
  &\qquad\qquad+\sum_{i=1}^b\biggl(s_0^{a+b-i}(1+(-1)^{a-1}c_1c_2)s_1s_0^{i}
  -s_0^{i-1}(1+(-1)^{a-1}c_1c_2)s_1s_0^{a+b-i+1}\biggr)\\
  &=\biggl(s_0^{b+1}s_1s_0^as_1+s_0^{a+b}(1+(-1)^{a-1}c_1c_2)s_1
  -s_0^b(1+(-1)^{a-1}c_1c_2)s_1s_0^{a}\biggr)\\
  &\qquad\qquad+\sum_{i=1}^b\biggl(s_0^{a+b-i}(1+(-1)^{a-1}c_1c_2)s_1s_0^{i}
  -s_0^{i-1}(1+(-1)^{a-1}c_1c_2)s_1s_0^{a+b-i+1}\biggr)\\
    &=\biggl(s_0^{b+1}s_1s_0^as_1+\sum_{i=1}^b \biggl(s_0^{a+b-i}(1+(-1)^{a-1}c_1c_2)s_1s_0^{i}\biggr)+s_0^{a+b}(1+(-1)^{a-1}c_1c_2)s_1\\
  &\qquad\qquad-\sum_{i=1}^b\biggl(s_0^{i-1}(1+(-1)^{a-1}c_1c_2)s_1s_0^{a+b-i+1}\biggr)-
  s_0^b(1+(-1)^{a-1}c_1c_2)s_1s_0^{a}\\
  &=s_0^{b+1}s_1s_0^as_1+\sum_{i=1}^{b+1}\biggl(s_0^{a+b+1-i}(1+(-1)^{a-1}c_1c_2)s_1s_0^{i-1}
  -s_0^{i-1}(1+(-1)^{a-1}c_1c_2)s_1s_0^{a+b-i+1}\biggr),
\end{align*} where in the first equation we have used induction hypothesis on $b$ and in the second equation we have used the result for $b=1$. This completes the proof of the Lemma.
\end{proof}


\begin{lemma}\label{expression}
Let $w\in W_{d,n}$. We fix two reduced expressions ${\bf w},{\bf w'}$ of $w$, and $I \subset [n]$. Then in $(\mhcn^{g})'$, we have \begin{equation}\label{bmforms}
{\bf w}c_I-{\bf w'}c_I\in \sum_{\substack{y\in W_{d,n}\\ \ell({\bf{y}})\leq\ell(w)-2\\ I'\subset [n]}} {\rm R} {\bf{y}}c_{I'},
\end{equation}
where the sum runs over all reduced expression ${\bf{y}}$ of elements in $W_{d,n}$ of length $\ell({\bf{y}})\leq\ell(w)-2$.
\end{lemma}

\begin{proof}
  By \cite[Lemma 1.5]{BM}, we have ${\bf w},{\bf w'}$ are weakly braid equivalent in $G(d,1,n)$. Now we substitute \eqref{brl0} and braid relations  by \eqref{weakly braid relation''} and \eqref{braid'} respectively in $\mhcn^{g}$, and use relations \eqref{sc1}, \eqref{sc2} if necessary. The Lemma follows.
\end{proof}

\begin{remark}
 Using \eqref{weakly braid relation''}, we can also deduce the analogue of \cite[Lemma 2.3]{BM} or \cite[Lemma 5.3]{HuSS} and then use this to prove the above Lemma as in  \cite[Lemma 2.4]{BM} or \cite[Lemma 5.4]{HuSS}.
\end{remark}

Now for each $w\in W_{d,n}$, we fix a reduced expression ${\bf w}$ of $w$.

\begin{theorem}\label{Degbasis} Let ${\rm R}$ be an integral domain. Then the set $\{{\bf w}c_I|w\in W_{d,n},\,I\subset [n]\}$ forms an ${\rm R}$-basis of $(\mhcn^{g})'$. Moreover, $\Psi$ is an isomorphism.
\end{theorem}

\begin{proof} Using induction on $\ell(w)$ and Lemma \ref{expression}, it is easy to see that the set $\{{\bf w}c_I|w\in W_{d,n},\,I\subset [n]\}$ generates $(\mhcn^{g})'$ as an ${\rm R}$-module. It remains to show that this is ${\rm R}$-linearly independent.
Let ${\rm F}$ be the fraction field of ${\rm R}$.
Set $(\mhcn^{g})'({\rm F}):={\rm F}\otimes_{R}(\mhcn^{g})', \mhcn^{g}({\rm F}):={\rm F}\otimes_{R}\mhcn^{g}$. Applying \eqref{basis1}, we know that $$
\dim \mhcn^{g}({\rm F})=2^n\#W_{d,n},
$$
it follows that the surjective map ${\rm F}\otimes_{{\rm R}}\Psi$ must be injective. In particular, $\{{\bf w}c_I|w\in W_{d,n},\,I\subset [n]\}$ is linearly independent over ${\rm F}$ hence linearly independent over ${\rm R}$ and $\Psi$ sends a basis to a basis. This completes the proof of the theorem.
\end{proof}

\subsection{A basis of ${\rm Tr}(\mhcn^{g})_{\overline{0}}$}
By Theorem \ref{Degbasis}, we can identify $\mhcn^{g}$ with $(\mhcn^{g})'$ and identify the generator $x_1\in \mhcn^{g}$ with $s_0\in (\mhcn^{g})'$.

\begin{lemma}\label{smallcommut} For $1\leq k\leq t\leq n-1,\,1\leq l\leq d-1$, the following holds in $\mhcn^{g}$.
 $$\begin{aligned} s_ks_{k-1}\cdots s_1 s^l_0s_1\cdots s_{t} c_i=\begin{cases}
  c_i s_ks_{k-1}\cdots s_1 s^l_0s_1\cdots s_{t} , & \mbox{if $1\leq i\leq k$ or $t+1<i\leq n$,} \\
  c_{i+1} s_ks_{k-1}\cdots s_1 s^l_0s_1\cdots s_{t} , & \mbox{if $k+1\leq i\leq t$,} \\
    (-1)^lc_{k+1} s_ks_{k-1}\cdots s_1 s^l_0s_1\cdots s_{t} , & \mbox{if $i=t+1$.}
                                                \end{cases}
    \end{aligned}$$
\end{lemma}
\begin{proof}
  This follows from relation \eqref{sc1}, \eqref{sc2}.
\end{proof}

Recall the definition of $\CP_n$. We define a subset $\widetilde{\CP_n}\subset \CP_n$ as follows:
\begin{equation}\label{label basis}
  \widetilde{\CP_n}:=\biggl\{(\lambda,\mu)\in\CP_n \biggm|\begin{matrix}
                                   \lambda_i +c(i)\equiv 1\pmod{2}, \, \forall 1\leq i\leq \ell(\lambda),\\
                                  \mu \text{ is an odd partition}.
                                 \end{matrix}\biggr\}
\end{equation} That is, $(\lambda,\mu)\in\CP_n$ is belong to $\widetilde{\CP_n}$ if and only if $\lambda_i$ is odd when the color $c(i)$ is even and $\lambda_i$ is even when the color $c(i)$ is odd and $\mu$ is an odd partition.

\begin{lemma}\label{span} Let ${\rm R}$ be any commutative unital ring with $2$ invertible. As an ${\rm R}$-module, we have \begin{equation}\label{generator1}
{\rm Tr}(\mhcn^{g})_{\overline{0}}=\text{\rm ${\rm R}$-Span}\bigl\{{\bf w}_\beta+[\mhcn^{g},\mhcn^{g}]_{\overline{0}}\bigm|\beta\in\widetilde{\CP_n}\bigr\} .
\end{equation}
\end{lemma}

\begin{proof}Set $$
\check{\mhcn^{g}}:=\text{\rm ${\rm R}$-Span}\bigl\{{\bf w}_\beta+[\mhcn^{g},\mhcn^{g}]_{\overline{0}}\bigm|\beta\in\widetilde{\CP_n}\bigr\} .
$$
We shall claim \begin{equation}
         {\bf w}c_{I} \in \check{\mhcn^{g}},\label{induction} \text{for any reduced expression  ${\bf w}$ of $w\in W_{d,n}$ and $I\in [n]$ with $|I|$ even}
       \end{equation} by induction upward on $\ell(w)$ and then upward on $|I|$.
 The case $\ell(w)=|I|=0$ is clear, since $1=w_{\alpha}$ where $\alpha=(\emptyset,(1^n))\in \widetilde{\CP_n}$. By induction hypothesis and Lemma \ref{expression}, it suffices to show that there exists one reduced  expression ${\bf w}$ of $w$ such that ${\bf w}c_I\in \check{\mhcn^{g}}$. The proof is divided into 5 steps as follows:

 \medskip
{\it Step 1.} We first show that inside ${\rm Tr}(\mhcn^{g})_{\overline{0}}$, ${\bf w}c_{I}$ can be ${\rm R}$-linearly spanned by some elements of form ${\bf w_{\rho,\varepsilon}}c_{\tilde{I}}$ together with some elements of the form ${\bf u}c_{I'}$ with $\ell(u)<\ell(w)$, where these $\rho=(\rho_1,\cdots,\rho_k)$ are compositions of $n$,  $\varepsilon=(\varepsilon_1,\cdots,\varepsilon_k)\in J^k$, $J:=\{0,1,\cdots,r-1\}$ and $\ell(w_{\rho,\varepsilon})\leq \ell(w),\,|\tilde{I}|=|I|$ . This is exactly the same as Step 1 in the proof of \cite[Theorem 4.3]{HS} and \cite[Theorem 5.9]{HuSS}, where we need to change \cite[Lemma 4.1]{HS} and \cite[Lemma 5.4]{HuSS} there by Lemma \ref{expression} here and use relation \eqref{sc1}, \eqref{sc2} if necessary and leave the other argument verbatim.

\smallskip
{\it Step 2.} We show that inside ${\rm Tr}(\mhcn^{g})_{\overline{0}}$, each of these ${\bf w_{\rho,\varepsilon}}c_{\tilde{I}}$ can be ${\rm R}$-linearly spanned by some elements of form ${\bf {w_\alpha}}c_{\tilde{\tilde{I}}}$  together with some elements of the form ${\bf u}c_{I'}$ with $\ell(u)<\ell(w)$, where these $\alpha$ are colored semi-bicompositions of $n$ and $\ell(w_\alpha)\leq \ell(w_{\rho,\varepsilon}),\,|\tilde{\tilde{I}}|=|\tilde{I}|$. This is exactly the same as Step 2 in the proof of \cite[Theorem 4.3]{HS} and \cite[Theorem 5.9]{HuSS}, where we need to change \cite[Lemma 4.1]{HS} and \cite[Lemma 5.4]{HuSS} there by Lemma \ref{expression} here and use relation \eqref{sc1}, \eqref{sc2} if necessary and leave the other argument verbatim.
\smallskip

{\it Step 3.} We show that for any colored semi-bicomposition $\alpha$ of $n$, each of these ${\bf {w_\alpha}}c_{\tilde{\tilde{I}}}$ can be ${\rm R}$-linearly spanned by some elements of form ${\bf {w_\beta}}c_{\overline{I}}$ together with some elements of the form ${\bf u}c_{I'}$ with $\ell(u)<\ell(w)$, where these $\beta$ are colored semi-bipartitions of $n$ and $\ell(w_\alpha)=\ell(w_{\beta}),\,|\overline{I}|=|\tilde{\tilde{I}}|$ . This is exactly the same as Step 3 in the proof of \cite[Theorem 4.3]{HS} and \cite[Theorem 5.9]{HuSS}, where we need to change \cite[Lemma 4.1]{HS} and \cite[Lemma 5.4]{HuSS} there by Lemma \ref{expression} here and use relation \eqref{sc1}, \eqref{sc2} if necessary and leave the other argument verbatim.
\smallskip

{\it Step 4.} We show that for any colored semi-bipartitions $\beta=(\lambda,\mu)$ of $n$,  if $|\overline{I}|\neq 0$, then each of these ${\bf {w_\beta}}c_{\overline{I}}$ can be ${\rm R}$-linearly spanned by some elements of form ${\bf {w_\beta}}c_{\overline{\overline{I}}}$ together with some elements of the form ${\bf u}c_{I'}$ with $\ell(u)<\ell(w)$, where $\overline{\overline{I}}=|\overline{I}|-2$.

We decompose $\overline{I}=\bigsqcup_{m=1}^{\ell(\lambda)+\ell(\mu)} \overline{I}_m$ such that $\overline{I}_{m}\subset [r_{m}+1,r_{m+1}]$, where $r_m$ is defined as in \eqref{ri} using the corresponding \eqref{identification}. Now suppose that there exists some $m$ such that $|\overline{I}_{m}|$ is odd. Note that $|\overline{I}|$ is even. We can find $k_1<k_2$ such that $|\overline{I}_{k_1}|$ and $|\overline{I}_{k_2}|$ are both odd and for any $k_1<k<k_2$ or $k<k_1$, $|\overline{I}_k|$ is even. We have
\begin{align*}
{\bf {w_\beta}}c_{\overline{I}}&\equiv \biggl(\prod_{i=1}^{\ell(\lambda)+\ell(\mu)}{\bf w}_{\overline{\beta},\epsilon,i}\biggr) c_{\overline{I}} \\
&\equiv \biggl(\prod_{i=1}^{k_2-1}{\bf w}_{\overline{\beta},\epsilon,i} \biggr) \biggl(\prod_{i=k_2+1}^{\ell(\lambda)+\ell(\mu)}{\bf w}_{\overline{\beta},\epsilon,i}\biggr){\bf w}_{\overline{\beta},\epsilon,k_2}
\biggl(\prod_{i=1}^{\ell(\lambda)+\ell(\mu)}c_{\overline{I}_i}\biggr)\\
&\equiv \biggl(\prod_{i=1}^{k_2-1}{\bf w}_{\overline{\beta},\epsilon,i}\biggr) \biggl(\prod_{i=k_2+1}^{\ell(\lambda)+\ell(\mu)}{\bf w}_{\overline{\beta},\epsilon,i}\biggr)
\biggl(\prod_{i=1}^{k_2-1}c_{\overline{I}_i}\biggr){\bf w}_{\overline{\beta},\epsilon,k_2}c_{\overline{I}_{k_2}}
\biggl(\prod_{i=k_2+1}^{\ell(\lambda)+\ell(\mu)}c_{\overline{I}_i}\biggr)\\
&\equiv \biggl(\prod_{i=1}^{k_2-1}{\bf w}_{\overline{\beta},\epsilon,i}\biggr) \biggl(\prod_{i=k_2+1}^{\ell(\lambda)+\ell(\mu)}{\bf w}_{\overline{\beta},\epsilon,i}\biggr)
\biggl(\prod_{i=1}^{k_2-1}c_{\overline{I}_i}\biggr)
\biggl(\prod_{i=k_2+1}^{\ell(\lambda)+\ell(\mu)}c_{\overline{I}_i}\biggr)
{\bf w}_{\overline{\beta},\epsilon,k_2}c_{\overline{I}_{k_2}}\\
&\equiv{\bf w}_{\overline{\beta},\epsilon,k_2}c_{\overline{I}_{k_2}}\biggl(\prod_{i=1}^{k_2-1}{\bf w}_{\overline{\beta},\epsilon,i} \biggr)\biggl(\prod_{i=k_2+1}^{\ell(\lambda)+\ell(\mu)}{\bf w}_{\overline{\beta},\epsilon,i}\biggr)
\biggl(\prod_{i=1}^{k_2-1}c_{\overline{I}_i}\biggr)
\biggl(\prod_{i=k_2+1}^{\ell(\lambda)+\ell(\mu)}c_{\overline{I}_i}\biggr)\\
&\equiv {\bf w}_{\overline{\beta},\epsilon,k_2}\biggl(\prod_{i=1}^{k_2-1}{\bf w}_{\overline{\beta},\epsilon,i} \biggr)\biggl(\prod_{i=k_2+1}^{\ell(\lambda)+\ell(\mu)}{\bf w}_{\overline{\beta},\epsilon,i}\biggr)c_{\overline{I}_{k_2}}
\biggl(\prod_{i=1}^{k_2-1}c_{\overline{I}_i}\biggr)
\biggl(\prod_{i=k_2+1}^{\ell(\lambda)+\ell(\mu)}c_{\overline{I}_i}\biggr)\\
&\equiv \biggl(\prod_{i=1}^{\ell(\lambda)+\ell(\mu)}{\bf w}_{\overline{\beta},\epsilon,i}\biggr)c_{\overline{I}_{k_2}}
\biggl(\prod_{i=1}^{k_2-1}c_{\overline{I}_i}\biggr)
\biggl(\prod_{i=k_2+1}^{\ell(\lambda)+\ell(\mu)}c_{\overline{I}_i}\biggr)\\
&\equiv-\biggl(\prod_{i=1}^{\ell(\lambda)+\ell(\mu)}{\bf w}_{\overline{\beta},\epsilon,i}\biggr) c_{\overline{I}}\\
&\equiv -{\bf {w_\beta}}c_{\overline{I}} \pmod{[\mhcn^{g},\mhcn^{g}]_{\overline{0}}+\sum_{\substack{\ell(u)<\ell(w)\\ I'\subset [1,n]}}{\rm R}{\bf u}c_{I'}},
\end{align*}
 where in the third, fourth and sixth equation, we have used Lemma \ref{smallcommut}; in the fourth and the last second equation, we have used \eqref{clifford}; and in the first, second, last third and last equation, we have used \eqref{bmforms} and $|\overline{I}|$ is even. Now use the assumption that $2\in {\rm R}$ is invertible, we have
 $$
 {\bf {w_\beta}}c_{\overline{I}}\in [\mhcn^{g},\mhcn^{g}]_{\overline{0}}+\sum_{\substack{\ell(u)<\ell(w)\\ I'\subset [1,n]}} {\rm R}{\bf u}c_{I'}.
$$
Hence we can assume all of $|\overline{I}_{m}|$ are even. If there exists some $k$ such that $|\overline{I}_{k}|\neq 0$, let $a=\min \overline{I}_{k}$ be the minimal index in $\overline{I}_{k},$ obverse that $a\neq r_{k+1}$.
Then we have
\begin{align*}
{\bf {w_\beta}}c_{\overline{I}}&\equiv \biggl(\prod_{i=1}^{\ell(\lambda)+\ell(\mu)}{\bf w}_{\overline{\beta},\epsilon,i}\biggr) \biggl(\prod_{i=1}^{\ell(\lambda)+\ell(\mu)}c_{\overline{I}_i}\biggr) \\
&\equiv \biggl(\prod_{i=1}^{\ell(\lambda)+\ell(\mu)}{\bf w}_{\overline{\beta},\epsilon,i}\biggr) \biggl(\prod_{i=1}^{k-1}c_{\overline{I}_i}\biggr)c_a c_{\overline{I}_k \backslash a}\biggl(\prod_{i=k+1}^{\ell(\lambda)+\ell(\mu)}c_{\overline{I}_i}\biggr)\\
&\equiv c_{a+1}\biggl(\prod_{i=1}^{\ell(\lambda)+\ell(\mu)}{\bf w}_{\overline{\beta},\epsilon,i}\biggr) \biggl(\prod_{i=1}^{k-1}c_{\overline{I}_i}\biggr) c_{\overline{I}_k \backslash a}\biggl(\prod_{i=k+1}^{\ell(\lambda)+\ell(\mu)}c_{\overline{I}_i}\biggr)\\
&\equiv\biggl(\prod_{i=1}^{\ell(\lambda)+\ell(\mu)}{\bf w}_{\overline{\beta},\epsilon,i}\biggr) \biggl(\prod_{i=1}^{k-1}c_{\overline{I}_i}\biggr) c_{\overline{I}_k \backslash a}\biggl(\prod_{i=k+1}^{\ell(\lambda)+\ell(\mu)}c_{\overline{I}_i}\biggr) c_{a+1}\\
&\equiv\biggl(\prod_{i=1}^{\ell(\lambda)+\ell(\mu)}{\bf w}_{\overline{\beta},\epsilon,i}\biggr) \biggl(\prod_{i=1}^{k-1}c_{\overline{I}_i}\biggr) c_{\overline{I}_k \backslash a}c_{a+1}\biggl(\prod_{i=k+1}^{\ell(\lambda)+\ell(\mu)}c_{\overline{I}_i}\biggr) \pmod{[\mhcn^{g},\mhcn^{g}]_{\overline{0}}+\sum_{\substack{\ell(u)<\ell(w)\\ I'\subset [1,n]}}{\rm R}{\bf u}c_{I'}}.
\end{align*}
Using relation \eqref{square}, \eqref{sc1}, we see that either $c_{\overline{I}_k \backslash a}c_{a+1}=\pm c_{I''}$ with $I''\subset [r_{m}+1,r_{m+1}],\,|I''|=|\overline{I}_k|-2$ or $c_{\overline{I}_k \backslash a}c_{a+1}=\pm c_{I''}$ with $I''\subset [r_{m}+1,r_{m+1}],\,|I''|=|\overline{I}_k|,\,\min I''>a$. In the second situation, we can repeat the computation above and after finite steps, we will arrive the first situation. However, in the first situation, put $\overline{\overline{I}}=\biggl(\bigsqcup_{m=1}^{k-1} \overline{I}_m\biggr) \bigsqcup I'' \bigsqcup\biggl(\bigsqcup_{m=k+1}^{\ell(\lambda)+\ell(\mu)} \overline{I}_m\biggr)$. This proves the claim of this step.

{\it Step 5.} We show that for any colored semi-bipartitions $\beta=(\lambda,\mu)$ of $n$, if $\beta\notin \widetilde{\CP_n}$ then ${\bf {w_\beta}}$ can be ${\rm R}$-linearly spanned by some elements of the form ${\bf u}c_{I'}$ with $\ell(u)<\ell(w)$.

 For some $1\leq i\leq \ell(\lambda)+\ell(\mu)$, suppose both of $\epsilon_i,\,\overline{\beta}_i$ are odd. We consider ${\bf {w_\beta}}c_{[r_{i}+1,r_{i+1}]}c^{-1}_{[r_{i}+1,r_{i+1}]}$, where $r_i,\epsilon_i$ is defined as in \eqref{ri}, \eqref{color}, using the corresponding \eqref{identification}. We have
 \begin{align*}
         {\bf {w_\beta}}&= {\bf {w_\beta}}c_{[r_{i}+1,r_{i+1}]}c^{-1}_{[r_{i}+1,r_{i+1}]}
         \equiv \biggl(\prod_{m=1}^{\ell(\lambda)+\ell(\mu)}
         {\bf w}_{\overline{\beta},\epsilon,m}\biggr)c_{[r_{i}+1,r_{i+1}]}c^{-1}_{[r_{i}+1,r_{i+1}]} \\
         &\equiv c^{-1}_{[r_{i}+1,r_{i+1}]}\biggl(\prod_{m=1}^{\ell(\lambda)+\ell(\mu)}{\bf
         w}_{\overline{\beta},\epsilon,m}\biggr)c_{[r_{i}+1,r_{i+1}]} \\
         &\equiv c_{r_{i+1}}\cdots c_{r_{i}+1}\biggl(\prod_{m=1}^{\ell(\lambda)+\ell(\mu)}{\bf w}_{\overline{\beta},\epsilon,m}\biggr) c_{r_{i}+1}c_{r_{i}+2}\cdots c_{r_{i+1}} \\
         &\equiv c_{r_{i+1}}\cdots c_{r_{i}+1} c_{r_{i}+2}c_{r_{i}+3}\cdots c_{r_{i+1}}\biggl(\prod_{i=1}^{\ell(\lambda)+\ell(\mu)}{\bf w}_{\overline{\beta},\epsilon,i}\biggr)c_{r_{i+1}} \\
         &\equiv c_{r_{i}+1}\biggl(\prod_{m=1}^{\ell(\lambda)+\ell(\mu)}{\bf w}_{\overline{\beta},\epsilon,m}\biggr)c_{r_{i+1}} \\
         &\equiv (-1)^{\epsilon_i} c_{r_{i}+1}^2\biggl(\prod_{m=1}^{\ell(\lambda)+\ell(\mu)}{\bf w}_{\overline{\beta},\epsilon,m}\biggr)\\
         &\equiv -{\bf {w_\beta}} \pmod{[\mhcn^{g},\mhcn^{g}]_{\overline{0}}+\sum_{\substack{\ell(u)<\ell(w)\\ I'\subset [1,n]}}{\rm R}{\bf u}c_{I'}},
 \end{align*}
 where in the last third and last second equation, we have used Lemma \ref{smallcommut} and in the last third equation, we have used relation \eqref{clifford} and the fact $\overline{\beta}_i$ is odd. This together with $2$ is invertible, implies that ${\bf {w_\beta}}\in [\mhcn^{g},\mhcn^{g}]_{\overline{0}}+\sum_{\substack{\ell(u)<\ell(w)\\ I'\subset [1,n]}}{\rm R}{\bf u}c_{I'}$.

 Similarly, if both of $\epsilon_i,\,\overline{\beta}_i$ are even, we have
 \begin{align*}
         {\bf {w_\beta}}&= {\bf {w_\beta}}c_{[r_{i}+1,r_{i+1}]}c^{-1}_{[r_{i}+1,r_{i+1}]}
         \equiv \biggl(\prod_{m=1}^{\ell(\lambda)+\ell(\mu)}{\bf w}_{\overline{\beta},\epsilon,m}\biggr)c_{[r_{i}+1,r_{i+1}]}c^{-1}_{[r_{i}+1,r_{i+1}]} \\
         &\equiv c^{-1}_{[r_{i}+1,r_{i+1}]}\biggl(\prod_{m=1}^{\ell(\lambda)+\ell(\mu)}{\bf w}_{\overline{\beta},\epsilon,m}\biggr)c_{[r_{i}+1,r_{i+1}]} \\
         &\equiv c_{r_{i+1}}\cdots c_{r_{i}+1}\biggl(\prod_{m=1}^{\ell(\lambda)+\ell(\mu)}{\bf w}_{\overline{\beta},\epsilon,m}\biggr) c_{r_{i}+1}c_{r_{i}+2}\cdots c_{r_{i+1}} \\
         &\equiv c_{r_{i+1}}\cdots c_{r_{i}+1} c_{r_{i}+2}c_{r_{i}+3}\cdots c_{r_{i+1}}\biggl(\prod_{m=1}^{\ell(\lambda)+\ell(\mu)}{\bf w}_{\overline{\beta},\epsilon,m}\biggr)c_{r_{i+1}} \\
         &\equiv -c_{r_{i}+1}\biggl(\prod_{m=1}^{\ell(\lambda)+\ell(\mu)}{\bf w}_{\overline{\beta},\epsilon,m}\biggr)c_{r_{i+1}} \\
         &\equiv (-1)^{\epsilon_i+1} c_{r_{i}+1}^2\biggl(\prod_{m=1}^{\ell(\lambda)+\ell(\mu)}{\bf w}_{\overline{\beta},\epsilon,m}\biggr)\\
         &\equiv -{\bf {w_\beta}} \pmod{[\mhcn^{g},\mhcn^{g}]_{\overline{0}}+\sum_{\substack{\ell(u)<\ell(w)\\ I'\subset [1,n]}}{\rm R}{\bf u}c_{I'}},
 \end{align*}
 where in the fifth and last second equation, we have used Lemma \ref{smallcommut} and in the last third equation, we have used relation \eqref{clifford} and the fact $\overline{\beta}_i$ is even. Again, this implies that ${\bf {w_\beta}}\in [\mhcn^{g},\mhcn^{g}]_{\overline{0}}+\sum_{\substack{\ell(u)<\ell(w)\\ I'\subset [1,n]}}{\rm R}{\bf u}c_{I'}$.
\end{proof}

We want to show that $\bigl\{{\bf w}_\beta+[\mhcn^{g},\mhcn^{g}]_{\overline{0}}\bigm|\beta\in\widetilde{\CP_n}\bigr\} $ does give a basis of ${\rm Tr}(\mhcn^{g})_{\overline{0}}$. Recall that we have fixed $d,m\in \NN$ and the definition of $\mathscr{P}^{\bullet,m}_n$.

\begin{lemma}\label{bij between partitions}
\begin{enumerate}
	\item Suppose $d=2m$ is even, then there is a bijection from the set $\widetilde{\CP_n}\subset \CP_n$ onto the set $\mathscr{P}^{\mathsf{0},m}_n.$
	\item Suppose $d=2m+1$ is odd, then there is a bijection from the set $\widetilde{\CP_n}\subset \CP_n$ onto the set $\mathscr{P}^{\mathsf{s},m}_n.$
	\end{enumerate}
\end{lemma}

\begin{proof}
  Clearly, there is bijection $\vartheta_n$ from the set $\CP_n$ onto the set $\mathscr{P}^{d}_{n}$ of $d$-partitions of $n$ such that \begin{enumerate}
\item the $1$-st component of $\vartheta_n(\lam,\mu)$ is $\mu$; and
\item for each $2\leq i\leq d$, the $i$-th component of $\vartheta_n(\lam,\mu)$ is the unique partition obtained by reordering the order of all the rows of $\lam$ colored by $i-1$.
\end{enumerate} Note that the image of $\widetilde{\CP_n}$ under $\vartheta_n$ is the set of  $d$-partitions of $n$ whose odd components correspond to odd partitions and even components correspond to even partitions. Since for any $a\in \NN$, there is a natural bijection $$\mathscr{P}_a \rightarrow \bigsqcup_{b+c=a}\mathscr{P}^{\text{ev}}_b\times \mathscr{P}^{\text{odd}}_c.
$$ This gives rise to a natural bijection from $\vartheta_n(\widetilde{\CP_n})$ on to $\mathscr{P}^{\mathsf{0},m}_n$ when $d$ is even and on to $\mathscr{P}^{\mathsf{s},m}_n$ when $d$ is odd.
\end{proof}

{\bf Proof of Theorem \ref{cocenter}}:
  We recall the generic cyclotomic Sergeev algebra $\mhcn^{h}$ defined over $\ZZ[\frac{1}{2}][Q_1,\cdots,Q_{m}]$ and ${\rm K}$ is the algebraic closure of the fraction filed of $\ZZ[\frac{1}{2}][Q_1,\cdots,Q_{m}]$. By Corollary \ref{semisimple2}, Lemma \ref{span} and Lemma \ref{bij between partitions}, we know that (\ref{generator1}) is a basis of ${\rm Tr}(\mhcn^h)_{\overline{0}}$ over ${\rm K}.$ Hence, (\ref{generator1}) is also linearly independent over $\ZZ[\frac{1}{2}][Q_1,\cdots,Q_{m}]$ by Proposition \ref{base change}. By Lemma \ref{span} again, we have that (\ref{generator1}) is a basis of ${\rm Tr}(\mhcn^h)_{\overline{0}}$ over $\ZZ[\frac{1}{2}][Q_1,\cdots,Q_{m}].$
  Now the result follows from the base change $\ZZ[\frac{1}{2}][Q_1,\cdots,Q_{m}]\rightarrow {\rm R}$ by Proposition \ref{base change}.
\qed

Using Proposition \ref{sym}, Theorem \ref{dengerate}  and Theorem \ref{cocenter}, we obtain the following.

\begin{corollary}\label{center}
	Suppose ${\rm R}$ is an integral domain with $2$ invertible. If the level $d=2m+1$ is odd, then ${\rm Z}(\mhcn^{g})_{\overline{0}}$ is a free ${\rm R}$-module. In particular, $\rank_{{\rm R}}{\rm Z}(\mhcn^{g})_{\overline{0}}=|\widetilde{\CP_n}|=|\mathscr{P}^{\mathsf{s},m}_n|.$
	\end{corollary}
	
\begin{remark} We remark that Corollary \ref{center} recovers Ruff's result on the rank of ${\rm Z}(\mhcn^{g})_{\overline{0}}$  \cite[Theorem 5.61]{Ru} when the level is odd. In fact, Ruff's index set $\mathcal{M}_{n}^{{\rm ev}}(d)$ ( \cite[Definition 5.54 (ii), Lemma 5.55]{Ru}) is defined as follows:
	\begin{equation}
		\mathcal{M}_{n}^{{\rm ev}}(d):=\biggl\{\undla=(\lambda^{(1)},\ldots,\lambda^{(d)})\in \mathscr{P}^{d}_{n} \biggm|\begin{matrix}
			(\lambda_j^{(i)}-1)d+i-1 \text{ are even }\\
			\text{ for all } 1\leq i \leq d, 1\leq j \leq \ell(\lambda^{(i)})
		\end{matrix}\biggr\}. \nonumber
	\end{equation}
	Therefore,
	if $d=2m+1,$ then by Lemma \ref{bij between partitions},
	\begin{equation}
		\mathcal{M}_{n}^{{\rm ev}}(d)
		=\biggl\{\undla=(\lambda^{(1)},\ldots,\lambda^{(d)})\in \mathscr{P}^{d}_{n} \biggm|\begin{matrix}
			\lambda^{(i)} \text{ is odd (even) partition }\\
			\text{ if $i$ is odd (even), } 1\leq i \leq d
		\end{matrix}\biggr\}\simeq \mathscr{P}^{\mathsf{s},m}_{n}. \nonumber
	\end{equation}
	\end{remark}
\section{On the super cocenter ${\rm Sup Tr}(\mhcn^{g})_{\overline{0}}$}
In this subsection, we will construct a family of linear generators for ${\rm Sup Tr}(\mhcn^{g})_{\overline{0}}.$

Recall the definition of $\CP_n$. We define another subset $\widehat{\CP_n}\subset \CP_n$ as follows:
\begin{equation}\label{label basis'}
  \widehat{\CP_n}:=\biggl\{\beta=(\lambda,\mu)\in\CP_n \biggm |
  \begin{matrix}
                               \bar{\beta_i} \neq \bar{\beta_j} \text{ if $\epsilon_i = \epsilon_j$ are both even, $ 1 \leq i\neq j \leq \ell (\lambda) + \ell (\mu)$}; \\
                                  \#\{1 \leq i \leq \ell (\lambda) + \ell (\mu) \mid \epsilon_i \text{ is even} \} \text{ is even}
  \end{matrix}\biggr\}.
\end{equation}  That is, $\beta=(\lambda,\mu)\in \widehat{\CP_n}$ if the parts corresponding to the same even color are distinct and the length of even color is even.

For each $\beta=(\lambda,\mu)\in \widehat{\CP_n},$ we define
$$ {\bf w}_{\beta}^{cl}:=\biggl(\prod_{i=1}^{\ell(\lambda)+\ell(\mu)} w_{\bar{\beta},\epsilon,i}\biggr)\biggl(\prod_{i=1}^{\ell(\lambda)+\ell(\mu)} c_{r_{i+1}}^{(1+(-1)^{\epsilon_i})/2} \biggr)\in \left(\mhcn^{g}\right)_{\bar{0}}.$$
Then we have
\begin{lemma}\label{span2} Let ${\rm R}$ be any commutative unital ring with $2$ invertible. As an ${\rm R}$-module, we have \begin{equation}\label{generator2}
{\rm SupTr}(\mhcn^{g})_{\overline{0}}=\text{\rm ${\rm R}$-Span}
\bigl\{ {\bf w}_{\beta}^{cl} + [\mhcn^{g},\mhcn^{g}]^{-}_{\overline{0}}\bigm | \beta \in \widehat{\CP_n} \bigr\} .
\end{equation}
\end{lemma}

\begin{proof}
Set $$
\widetilde{\mhcn^{g}}:=\text{\rm ${\rm R}$-Span}
\bigl\{ {\bf w}_{\beta}^{cl} + [\mhcn^{g},\mhcn^{g}]^{-}_{\overline{0}}\bigm | \beta \in \widehat{\CP_n} \bigr\} .
$$
We shall claim \begin{equation}
{\bf w}c_{I} \in \widetilde{\mhcn^{g}},\label{induction1} \text{for any reduced expression  ${\bf w}$ of $w\in W_{d,n}$ and $I\in [n]$ with $|I|$ even}
\end{equation} by induction upward on $\ell(w)$ and then upward on $|I|$.
The case $\ell(w)=|I|=0$ is again clear. By induction hypothesis and Lemma \ref{expression}, it suffices to show that there exists one reduced  expression ${\bf w}$ of $w$ such that ${\bf w}c_I\in \widetilde{\mhcn^{g}}$.

Following  completely similar {\it Steps 1, 2, 3} in the proof of Lemma \ref{span}, we only need to consider
$${\bf w}={\bf w}_{\beta},\, \beta=(\lambda,\mu)\in \CP_n .$$ The remaing proof is divided into 4 steps as follows:

{\it Step 1.} Suppose $I=\bigsqcup_{i=1}^{\ell(\lambda)+\ell(\mu)} I_i$ such that $I_{i}\subset [r_{i}+1,r_{i+1}]$, where $r_i$ is defined as in \eqref{ri} using the corresponding \eqref{identification}, we claim that ${\bf w}_{\beta} c_I$ can be ${\rm R}$-linearly spanned by some element of form ${\bf w}_{\beta} c_{\bar{I}}$, where $\bar{I}_{i}\subset \{r_{i+1}\}$, $|\bar{I}_i|\leq |I_i|$ and $|\bar{I}|$ is even, together with some elements of the form ${\bf u}c_{I'}$ with $\ell({\bf u})<\ell({\bf w}_{\beta})$. Actually, if $I_i=\emptyset,\forall 1\leq i\leq \ell(\lambda)+\ell(\mu)$, we are done. Otherwise, assume $I_m\neq \emptyset$ for some $m$ and $a={\text {min} } I_m$ be the minimal index in $I_m$. We may assume $a\neq r_{m+1}$, otherwise we are done. Using a similar computation in the second part of {\it Step 4} of the proof of Lemma \ref{span}, we have
\begin{align*}
	{\bf {w_\beta}}c_{\overline{I}}&\equiv \biggl(\prod_{i=1}^{\ell(\lambda)+\ell(\mu)}{\bf w}_{\overline{\beta},\epsilon,i}\biggr) \biggl(\prod_{i=1}^{\ell(\lambda)+\ell(\mu)}c_{I_i}\biggr) \\
	&\equiv \biggl(\prod_{i=1}^{\ell(\lambda)+\ell(\mu)}{\bf w}_{\overline{\beta},\epsilon,i}\biggr) \biggl(\prod_{i=1}^{m-1}c_{I_i}\biggr)c_a c_{I_m \backslash a}\biggl(\prod_{i=m+1}^{\ell(\lambda)+\ell(\mu)}c_{I_i}\biggr)\\
	&\equiv \pm c_{a+1}\biggl(\prod_{i=1}^{\ell(\lambda)+\ell(\mu)}{\bf w}_{\overline{\beta},\epsilon,i}\biggr) \biggl(\prod_{i=1}^{m-1}c_{I_i}\biggr) c_{I_m \backslash a}\biggl(\prod_{i=m+1}^{\ell(\lambda)+\ell(\mu)}c_{I_i}\biggr)\\
	&\equiv \pm \biggl(\prod_{i=1}^{\ell(\lambda)+\ell(\mu)}{\bf w}_{\overline{\beta},\epsilon,i}\biggr) \biggl(\prod_{i=1}^{m-1}c_{I_i}\biggr) c_{\overline{I}_m \backslash a}\biggl(\prod_{i=m+1}^{\ell(\lambda)+\ell(\mu)}c_{I_i}\biggr) c_{a+1}\\
	&\equiv \pm\biggl(\prod_{i=1}^{\ell(\lambda)+\ell(\mu)}{\bf w}_{\overline{\beta},\epsilon,i}\biggr) \biggl(\prod_{i=1}^{m-1}c_{I_i}\biggr) c_{I_m \backslash a}c_{a+1}\biggl(\prod_{i=m+1}^{\ell(\lambda)+\ell(\mu)}c_{I_i}\biggr) \pmod{[\mhcn^{g},\mhcn^{g}]_{\overline{0}}^{-}+\sum_{\substack{\ell(u)<\ell(w)\\ I'\subset [1,n]}}{\rm R}{\bf u}c_{I'}}.
\end{align*}

Again, using relation \eqref{square}, \eqref{sc1}, we can do induction upward on $I_m$ and then downward on $a={\text {min} } I_m$ to complete the claim of this step as in the second part of {\it Step 4} in the proof of Lemma \ref{span}.

{\it Step 2.}  Suppose $\bar{I}=\bigsqcup_{i=1}^{\ell(\lambda)+\ell(\mu)} \bar{I}_i$ such that $\bar{I}_{i}\subset [r_{i}+1,r_{i+1}]$, where $r_i$ is defined as in \eqref{ri} using the corresponding \eqref{identification}. We show that for any colored semi-bipartitions $\beta=(\lambda,\mu)$ of $n$, if $\epsilon_m$ is odd and $\{r_{m+1}\} =\bar{I}_m$, then $ {\bf w}_{\beta}c_{\bar{I}}$ can be ${\rm R}$-linearly spanned by some elements of the form ${\bf u}c_{I'}$ with $\ell(u)<\ell(w)$.

We write $\bar{I}=\biggl(\bigsqcup_{i=1}^{m-1} \bar{I}_i\biggr) \bigsqcup \{r_{m+1}\} \bigsqcup\biggl(\bigsqcup_{i=m+1}^{\ell(\lambda)+\ell(\mu)} \bar{I}_i\biggr),$ where $\bar{I}_{i}=\emptyset$ or $\{r_{i+1}\},$
and we assume that $\biggm|\bigsqcup_{i=1}^{m-1} \bar{I}_i \biggm|$ is even, $\biggm| \bigsqcup_{i=m+1}^{\ell(\lambda)+\ell(\mu)} \bar{I}_i \biggm|$ is odd.
Then we have
\begin{align*}
{\bf {w_\beta}}c_{\bar{I}}&\equiv \biggl(\prod_{i=1}^{\ell(\lambda)+\ell(\mu)}{\bf w}_{\overline{\beta},\epsilon,m}\biggr) \biggl(\prod_{i=1}^{\ell(\lambda)+\ell(\mu)}c_{\bar{I}_{i}}\biggr) \qquad &\\
&\equiv \biggl(\prod_{i=1}^{\ell(\lambda)+\ell(\mu)}{\bf w}_{\overline{\beta},\epsilon,i}\biggr) \biggl(\prod_{i=1}^{m-1}c_{\bar{I}_{i}}\biggr)c_{r_{m+1}} \biggl(\prod_{i=m+1}^{\ell(\lambda)+\ell(\mu)}c_{\bar{I}_{i}}\biggr)\qquad &\\
&\equiv (-1)^{\epsilon_m}c_{r_{m}+1}\biggl(\prod_{i=1}^{\ell(\lambda)+\ell(\mu)}{\bf w}_{\overline{\beta},\epsilon,i}\biggr) \biggl(\prod_{i=1}^{m-1}c_{\bar{I}_{i}}\biggr) \biggl(\prod_{i=m+1}^{\ell(\lambda)+\ell(\mu)}c_{\bar{I}_{i}}\biggr) \qquad &\Biggl({\text {$\biggm|\bigsqcup_{i=1}^{m-1} \bar{I}_i \biggm|$ is even and Lemma \ref{smallcommut}}}\Biggr)\\
&\equiv -c_{r_{m}+1}\biggl(\prod_{i=1}^{\ell(\lambda)+\ell(\mu)}{\bf w}_{\overline{\beta},\epsilon,m}\biggr) \biggl(\prod_{i=1}^{m-1}c_{\bar{I}_{i}}\biggr) \biggl(\prod_{i=m+1}^{\ell(\lambda)+\ell(\mu)}c_{\bar{I}_{i}}\biggr)\qquad &\biggl({\text {$\epsilon_m$ is odd}}\biggr)\\
&\equiv\biggl(\prod_{i=1}^{\ell(\lambda)+\ell(\mu)}{\bf w}_{\overline{\beta},\epsilon,i}\biggr) \biggl(\prod_{i=1}^{m-1}c_{\bar{I}_{i}}\biggr) \biggl(\prod_{i=m+1}^{\ell(\lambda)+\ell(\mu)}c_{\bar{I}_{i}}\biggr)c_{r_{m}+1}\qquad &\\
&\equiv -\biggl(\prod_{i=1}^{\ell(\lambda)+\ell(\mu)}{\bf w}_{\overline{\beta},\epsilon,i}\biggr) \biggl(\prod_{i=1}^{m-1}c_{\bar{I}_{i}}\biggr) c_{r_{m}+1} \biggl(\prod_{i=m+1}^{\ell(\lambda)+\ell(\mu)}c_{\bar{I}_{i}}\biggr)\qquad &\Biggl({\text{ $\biggm| \bigsqcup_{i=m+1}^{\ell(\lambda)+\ell(\mu)} \bar{I}_i \biggm|$ is odd}}\Biggr)  \\
&\equiv -c_{r_{m}+2}\biggl(\prod_{i=1}^{\ell(\lambda)+\ell(\mu)}{\bf w}_{\overline{\beta},\epsilon,i}\biggr) \biggl(\prod_{i=1}^{m-1}c_{\bar{I}_{i}}\biggr) \biggl(\prod_{i=m+1}^{\ell(\lambda)+\ell(\mu)}c_{\bar{I}_{i}}\biggr)\qquad &\Biggl({\text {$\biggm|\bigsqcup_{i=1}^{m-1} \bar{I}_i \biggm|$ is even and Lemma \ref{smallcommut}}}\Biggr)\\
&\equiv \biggl(\prod_{i=1}^{\ell(\lambda)+\ell(\mu)}{\bf w}_{\overline{\beta},\epsilon,i}\biggr) \biggl(\prod_{i=1}^{m-1}c_{\bar{I}_{i}}\biggr) \biggl(\prod_{i=m+1}^{\ell(\lambda)+\ell(\mu)}c_{\bar{I}_{i}}\biggr)c_{r_{m}+2}\qquad &\\
&\equiv -\biggl(\prod_{i=1}^{\ell(\lambda)+\ell(\mu)}{\bf w}_{\overline{\beta},\epsilon,i}\biggr) \biggl(\prod_{i=1}^{m-1}c_{\bar{I}_{i}}\biggr) c_{r_{m}+2} \biggl(\prod_{i=m+1}^{\ell(\lambda)+\ell(\mu)}c_{\bar{I}_{i}}\biggr) \qquad &\\
&\equiv \cdots \qquad &\\
&\equiv -\biggl(\prod_{i=1}^{\ell(\lambda)+\ell(\mu)}{\bf w}_{\overline{\beta},\epsilon,i}\biggr) \biggl(\prod_{i=1}^{m-1}c_{\bar{I}_{i}}\biggr) c_{r_{m+1}} \biggl(\prod_{i=m+1}^{\ell(\lambda)+\ell(\mu)}c_{\bar{I}_{i}}\biggr)\qquad &\\
&\equiv -{\bf {w_\beta}}c_{\bar{I}}
 \pmod{[\mhcn^{g},\mhcn^{g}]^{-}_{\overline{0}}+\sum_{\substack{\ell(u)<\ell(w)\\ I'\subset [1,n]}}{\rm R}{\bf u}c_{I'}}.\qquad &
\end{align*}
 This together with $2$ is invertible, implies that ${\bf {w_\beta}}c_{\bar{I}}\in [\mhcn^{g},\mhcn^{g}]^{-}_{\overline{0}}+\sum_{\substack{\ell(u)<\ell(w)\\ I'\subset [1,n]}}{\rm R}{\bf u}c_{I'}$. If $\biggm|\bigsqcup_{i=1}^{m-1} \bar{I}_i \biggm|$ is odd, and $\biggm| \bigsqcup_{i=m+1}^{\ell(\lambda)+\ell(\mu)} \bar{I}_i \biggm|$ is even, the computation is similar.

{\it Step 3.}  Suppose $\bar{I}=\bigsqcup_{i=1}^{\ell(\lambda)+\ell(\mu)} \bar{I}_i$ such that $I_{i}\subset [r_{i}+1,r_{i+1}]$, where $r_i$ is defined as in \eqref{ri} using the corresponding \eqref{identification}. We show that for any colored semi-bipartitions $\beta=(\lambda,\mu)$ of $n$, if $\epsilon_m$ is even, $\bar{I}_m=\emptyset$, then $ {\bf w}_{\beta}c_{\bar{I}}$ can be ${\rm R}$-linearly spanned by some elements of the form ${\bf u}c_{I'}$ with $\ell(u)<\ell(w)$. In fact,
\begin{align*}
	{\bf {w_\beta}}c_{\bar{I}} & = {\bf {w_\beta}}c_{\bar{I}}c_{r_{m}+1}c_{r_{m}+1}\qquad &\\
		&\equiv -c_{r_{m}+1} \biggl(\prod_{i=1}^{\ell(\lambda)+\ell(\mu)}{\bf w}_{\overline{\beta},\epsilon,i}\biggr)
		c_{\bar{I}}c_{r_{m}+1} \qquad &\\
		&\equiv -(-1)^{\epsilon_m} \biggl(\prod_{i=1}^{\ell(\lambda)+\ell(\mu)}{\bf w}_{\overline{\beta},\epsilon,i}\biggr)
		c_{r_{m+1}} c_{\bar{I}}c_{r_{m}+1}\qquad & \biggl({\text{Lemma \ref{smallcommut}}}\biggr) \\
		&\equiv - \biggl(\prod_{i=1}^{\ell(\lambda)+\ell(\mu)}{\bf w}_{\overline{\beta},\epsilon,i}\biggr)
		c_{r_{m+1}} c_{\bar{I}}c_{r_{m}+1} \qquad&\Biggl( {\text{$\eps_m$ is even}}\Biggr)\\
		&\equiv \biggl(\prod_{i=1}^{\ell(\lambda)+\ell(\mu)}{\bf w}_{\overline{\beta},\epsilon,i}\biggr)
		c_{\bar{I}}c_{r_{m}+1}c_{r_{m+1}} \qquad &\\
		&\equiv -c_{r_{m+1}} \biggl(\prod_{i=1}^{\ell(\lambda)+\ell(\mu)}{\bf w}_{\overline{\beta},\epsilon,i}\biggr)
		c_{\bar{I}}c_{r_{m}+1} \qquad &\\
		&\equiv - \biggl(\prod_{i=1}^{\ell(\lambda)+\ell(\mu)}{\bf w}_{\overline{\beta},\epsilon,m}\biggr)c_{r_{m+1}-1}
		c_{\bar{I}}c_{r_{m}+1} \qquad & \biggl({\text{Lemma \ref{smallcommut}}}\biggr)\\
		&\equiv \cdots\qquad &\\
	&\equiv -\biggl(\prod_{i=1}^{\ell(\lambda)+\ell(\mu)}{\bf w}_{\overline{\beta},\epsilon,i}\biggr)
	c_{r_{m}+1}c_{\bar{I}}c_{r_{m}+1}\qquad &\\
	&\equiv -\biggl(\prod_{i=1}^{\ell(\lambda)+\ell(\mu)}{\bf w}_{\overline{\beta},\epsilon,i}\biggr)
	c_{\bar{I}}c_{r_{m}+1}^{2}\qquad &\Biggl({\text {$|\bar{I}|$ is even}}\Biggr)\\
	&\equiv -{\bf {w_\beta}}c_{\bar{I}}
	\pmod{[\mhcn^{g},\mhcn^{g}]^{-}_{\overline{0}}+\sum_{\substack{\ell(u)<\ell(w)\\ I'\subset [1,n]}}{\rm R}{\bf u}c_{I'}}. \qquad &
	\end{align*}  This together with $2$ is invertible, again implies that ${\bf {w_\beta}}c_{\bar{I}}\in [\mhcn^{g},\mhcn^{g}]^{-}_{\overline{0}}+\sum_{\substack{\ell(u)<\ell(w)\\ I'\subset [1,n]}}{\rm R}{\bf u}c_{I'}$.

{\it Step 4.}  Suppose $\bar{I}=\bigsqcup_{i=1}^{\ell(\lambda)+\ell(\mu)} \bar{I}_i$ such that $I_{i}\subset [r_{i}+1,r_{i+1}]$, where $r_i$ is defined as in \eqref{ri} using the corresponding \eqref{identification}. By {\it Step 3, 4}, we may assume that $\bar{I}_i=\{r_{i+1}\}$ if $\epsilon_i$ is even, and $\bar{I}_i=\emptyset$ if $\epsilon_i$ is odd. We show that for any colored semi-bipartitions $\beta=(\lambda,\mu)$ of $n$, if $\epsilon_i = \epsilon_{j}$ is even and $\bar{\beta}_i= \bar{\beta}_{j}$ for $1\leq i<j \leq \ell(\lambda)+\ell(\mu),$ then $ {\bf w}_{\beta}c_{\bar{I}}$ can be ${\rm R}$-linearly spanned by some elements of the form ${\bf u}c_{I'}$ with $\ell(u)<\ell(w)$.

By assumption, we have $\bar{I}_i=\{r_{i+1}\},$ $\bar{I}_{j}=\{r_{j+1}\}$ since $\epsilon_i = \epsilon_{j}$ is even. We set
$$v:= \prod_{k=1}^{\bar{\beta}_i}\biggl( s_{r_{j}+k-1}\cdots s_{r_{i}+k+1}s_{r_{i}+k}s_{r_{m}+k+1}\cdots s_{r_{j}+k-1}\biggl)\in \left(\mhcn^{g}\right)_{\bar{0}}.$$
As in the second step of the proof of \cite[Theorem 4.3]{HS}, we can prove
\begin{align*}
v{\bf w}_{\beta}v^{-1}
&\equiv v\prod_{m=1}^{\ell(\lambda)+\ell(\mu)}{\bf w}_{\overline{\beta},\epsilon,m} v^{-1}\\
&\equiv v\biggl( \prod_{m=1}^{i-1} {\bf w}_{\overline{\beta},\epsilon,m}\biggr)v^{-1} \left(v{\bf w}_{\overline{\beta},\epsilon,i}v^{-1}\right) v\biggl( \prod_{m=i+1}^{j-1} {\bf w}_{\overline{\beta},\epsilon,m}\biggr)v^{-1} \left(v{\bf w}_{\overline{\beta},\epsilon,j}v^{-1}\right)
v\biggl( \prod_{m=j+1}^{\ell(\lambda)+\ell(\mu)} {\bf w}_{\overline{\beta},\epsilon,m}\biggr)v^{-1}\\
&\equiv \biggl( \prod_{m=1}^{i-1} {\bf w}_{\overline{\beta},\epsilon,m}\biggr) {\bf w}_{\overline{\beta},\epsilon,j} \biggl( \prod_{m=i+1}^{j-1} {\bf w}_{\overline{\beta},\epsilon,m}\biggr) {\bf w}_{\overline{\beta},\epsilon,i}
\biggl( \prod_{m=j+1}^{\ell(\lambda)+\ell(\mu)} {\bf w}_{\overline{\beta},\epsilon,m}\biggr)\\
&\equiv {\bf w}_{\beta}
\pmod{[\mhcn^{g},\mhcn^{g}]^{-}_{\overline{0}}+\sum_{\substack{\ell(u)<\ell(w)\\ I'\subset [1,n]}}{\rm R}{\bf u}c_{I'}}.
\end{align*}
Thus we have
\begin{align*}
{\bf w}_{\beta}c_{\bar{I}}
&\equiv v{\bf w}_{\beta}c_{\bar{I}}v^{-1}\\
&\equiv v{\bf {w_\beta}}v^{-1}v\biggl(\prod_{m=1}^{i-1}c_{\bar{I}_{m}}\biggr)v^{-1} \biggl(vc_{r_{i+1}}v^{-1} \biggr) v\biggl(\prod_{m=i+1}^{j-1}c_{\bar{I}_{m}}\biggr)v^{-1} \biggl( vc_{r_{j+1}}v^{-1}\biggr) v\biggl(\prod_{i=j+1}^{\ell(\lambda)+\ell(\mu)}c_{\bar{I}_{m}}\biggr)v^{-1}\\
&\equiv {\bf {w_\beta}}\biggl(\prod_{m=1}^{i-1}c_{\bar{I}_{m}}\biggr) c_{r_{j+1}} \biggl(\prod_{m=i+1}^{j-1}c_{\bar{I}_{m}}\biggr) c_{r_{i+1}} \biggl(\prod_{i=j+1}^{\ell(\lambda)+\ell(\mu)}c_{\bar{I}_{m}}\biggr)\\
&\equiv -{\bf {w_\beta}}\biggl(\prod_{m=1}^{i-1}c_{\bar{I}_{m}}\biggr) c_{r_{i+1}} \biggl(\prod_{m=i+1}^{j-1}c_{\bar{I}_{m}}\biggr) c_{r_{j+1}} \biggl(\prod_{i=j+1}^{\ell(\lambda)+\ell(\mu)}c_{\bar{I}_{m}}\biggr)\\
&\equiv -{\bf w}_{\beta}c_{\bar{I}}
\pmod{[\mhcn^{g},\mhcn^{g}]^{-}_{\overline{0}}+\sum_{\substack{\ell(u)<\ell(w)\\ I'\subset [1,n]}}{\rm R}{\bf u}c_{I'}}.
\end{align*}
This together with $2$ is invertible, implies that ${\bf {w_\beta}}c_{\bar{I}}\in [\mhcn^{g},\mhcn^{g}]^{-}_{\overline{0}}+\sum_{\substack{\ell(u)<\ell(w)\\ I'\subset [1,n]}}{\rm R}{\bf u}c_{I'}$.
\end{proof}

Combining Proposition \ref{sym}, Theorem \ref{dengerate} and  Lemma \ref{span2}, we obtain the following result on the centers of cyclotomic Sergeev algbras of even level.
\begin{corollary}\label{upperbound}Suppose  $F$ is any field with ${\rm char F}\neq 2.$ If the level $d$ is even, then $\dim_{{\rm F}}{\rm Z}(\mhcn^{g})_{\overline{0}}\leq |\widehat{\CP_n}|.$
\end{corollary}
Recall  the definition of $\mathscr{MP}^{\bullet,m}_{n}$ \eqref{simple modules of type M}.
\begin{conjecture}\label{conj}
	Suppose ${\rm R}$ is an integral domain with $2$ invertible. Then ${\rm SupTr}(\mhcn^{g})_{\overline{0}}$ is a free ${\rm R}$-module. Moreover,  $$\rank_{{\rm R}}{\rm SupTr}(\mhcn^{g})_{\overline{0}}=\begin{cases}|\mathscr{MP}^{\mathsf{0},m}_{n}|, &\text{if $d$ is even,}\\
		 |\mathscr{MP}^{\mathsf{s},m}_{n}|, &\text{if $d$ is odd.}
		\end{cases}$$
	\end{conjecture}

\begin{theorem} \label{supercocenter}
	For $d=1$, the above Conjecture \ref{conj} holds.
	\end{theorem}
	
	\begin{proof}
		When $d=1,$ then $m=0$ and we have $\widehat{\CP_n}$ is the set of strict partition of $n$ with even length which exactly equals to $\mathscr{MP}^{\mathsf{s},0}_{n}$ by definition. The remaining proof is the same as Theorem \ref{cocenter}.
		\end{proof}

\providecommand{\bysame}{\leavevmode\hbox to3em{\hrulefill}\thinspace}
\providecommand{\MR}{\relax\ifhmode\unskip\space\fi MR }
\providecommand{\MRhref}[2]{%

\href{http://www.ams.org/mathscinet-getitem?mr=#1}{#2} }
\providecommand{\href}[2]{#2}

\end{document}